\documentclass[12pt]{amsart}
\title{Large implies Henselian}
\author{Will Johnson, Chieu-Minh Tran, Erik Walsberg, and Jinhe Ye}
\email{willjohnson@fudan.edu.cn, trancm@nus.edu.sg, erik.walsberg@gmail.com, jinhe.ye.maths@gmail.com}

\usepackage{amsmath, amssymb, amsthm, enumitem, comment}    
\usepackage[mathscr]{euscript}
\usepackage[Symbol]{upgreek}
\usepackage{fullpage} 	
\usepackage{hyperref}
\usepackage{lineno}
\usepackage[all]{xy}
\usepackage{centernot}
\usepackage{tikz-cd}
\usepackage{xspace}
\usepackage{enumitem}
\usepackage{cleveref}

\DeclareFontFamily{U}{fsy}{}
\DeclareFontShape{U}{fsy}{m}{n}{<->s*[.9]psyr}{}
\DeclareSymbolFont{der@m}{U}{fsy}{m}{n}
\DeclareMathSymbol{\der}{\mathord}{der@m}{182}

\DeclareFontFamily{U}{BOONDOX-calo}{\skewchar\font=45 }
\DeclareFontShape{U}{BOONDOX-calo}{m}{n}{
  <-> s*[1.05] BOONDOX-r-calo}{}
\DeclareFontShape{U}{BOONDOX-calo}{b}{n}{
  <-> s*[1.05] BOONDOX-b-calo}{}
\DeclareMathAlphabet{\mathcalboondox}{U}{BOONDOX-calo}{m}{n}
\SetMathAlphabet{\mathcalboondox}{bold}{U}{BOONDOX-calo}{b}{n}
\DeclareMathAlphabet{\mathbcalboondox}{U}{BOONDOX-calo}{b}{n}

\DeclareSymbolFont{imag@m}{OT1}{cmr}{m}{ui}
\DeclareMathSymbol{\imag}{\mathord}{imag@m}{105}

\DeclareMathOperator*{\forkindep}{\raise0.2ex\hbox{\ooalign{\hidewidth$\vert$\hidewidth\cr\raise-0.9ex\hbox{$\smile$}}}}

\newcommand{\Sa}[1]{\ensuremath{\mathscr{#1}}}

\newcommand{\perf}{K^{\mathrm{perf}}}
\newcommand{\lowenheim}{L\"owenheim-Skolem\xspace}
\newcommand{\eval}{\operatorname{Eval}}
\newcommand{\st}{\operatorname{st}}

\newcommand{\pring}{K[[t_1,\ldots,t_m]]}

\newcommand{\pringI}{K[[t_i]]_{i\in I}}
\newcommand{\palg}{K[t_1,\ldots,t_m]^\mathrm{h}}

\newcommand{\palgbasn}{K[[t_i]]^\mathrm{h}}
\newcommand{\palgkap}{\palgbasn_{i < \kappa}}
\newcommand{\palgI}{K[t_i]^\mathrm{h}_{i \in I}}

\newcommand{\kalg}{K^{\mathrm{alg}}}

\newcommand{\Gal}{\operatorname{Gal}}
\newcommand{\Spec}{\operatorname{Spec}}

\newcommand{\Frac}{\operatorname{Frac}}

\newcommand{\jac}{\operatorname{Jac}}

\newcommand{\red}{\mathrm{red}}

\newcommand{\poly}{\mathrm{Poly}}
\newcommand{\id}{\operatorname{id}}
\newcommand{\prc}{\mathrm{PRC}}
\newcommand{\pac}{\mathrm{PAC}}
\newcommand{\alg}{{\mathrm{alg}}}
\newcommand{\sep}{{\mathrm{sep}}}
\newcommand{\C}{\mathcal{C}}

\newtheorem*{Claim*}{Claim}
\newtheorem{theorem}{Theorem}[section] 
\newtheorem{lemma}[theorem]{Lemma}

\newtheorem{prop-def}[theorem]{Proposition-Definition}
\newtheorem{corollary}[theorem]{Corollary}
\newtheorem{fact}[theorem]{Fact}
\newtheorem{fact-eh}[theorem]{Fact(?)}

\newtheorem{conjecture}[theorem]{Conjecture}

\newtheorem{question}[theorem]{Question}
\newtheorem{proposition}[theorem]{Proposition}
\newtheorem{proposition-eh}[theorem]{Proposition(?)}
\newtheorem*{theorem-star}{Theorem}
\newtheorem*{conjecture-star}{Conjecture}
\newtheorem*{lemma-star}{Lemma}

\newtheorem*{thmA}{Theorem A}
\newtheorem*{thmB}{Theorem B}
\newtheorem*{thmC}{Theorem C}
\newtheorem*{thmD}{Theorem D}

\newtheorem*{lemma6.2alt}{Lemma \ref*{confusing-v2}$'$}

\theoremstyle{definition}
\newtheorem{definition}[theorem]{Definition}

\newtheorem{remark}[theorem]{Remark}
\theoremstyle{remark}
\newtheorem{claim}[theorem]{Claim}

\newcommand{\Aa}{\mathbb{A}}
\newcommand{\Ff}{\mathbb{F}}

\newcommand{\Qq}{\mathbb{Q}}
\newcommand{\Rr}{\mathbb{R}}
\newcommand{\Kk}{\mathbb{K}}
\newcommand{\Zz}{\mathbb{Z}}
\newcommand{\Nn}{\mathbb{N}}
\newcommand{\Cc}{\mathbb{C}}
\newcommand{\Pp}{\mathbb{P}}
\newcommand{\mm}{\mathfrak{m}}

\newcommand{\cE}{\mathscr{E}}

\newcommand{\cF}{\mathscr{F}}
\newcommand{\cM}{\mathscr{M}}
\newcommand{\cN}{\mathscr{N}}

\newcommand{\cT}{\mathscr{T}}

\newenvironment{claimproof}[1][\proofname]
               {
                 \proof[#1]
                 
               }
               {
                 \endproof
               }

\begin{document}

\begin{abstract}
Fix a field $K$.
We show that $K$ is large if and only if some elementary extension of $K$ is the fraction field of a henselian local domain which is not a field.
The proof uses a new result about the \'etale-open topology over $K$: if $K$ is not separably closed and $V \to W$ is an \'etale morphism of $K$-varieties then $V(K) \to W(K)$ is a local homeomorphism in the \'etale-open topology.
This, in turn, follows from results comparing the \'etale-open topology on $V(K)$ and the finite-closed topology on $V(K)$, newly introduced in this paper.
We show that the \'etale-open topology refines the finite-closed topology when $K$ is perfect, and that the finite-closed topology refines the \'etale-open topology when $K$ is bounded.
It follows that these two topologies agree in many natural examples.
On the other hand, we construct several examples where these two differ, which allows us to answer a question of Lampe.
\end{abstract}

\maketitle

\section{Introduction} 
Let $K$ be a field.
Then $K$ is \textbf{large} if one of the following equivalent conditions holds.
\begin{enumerate}[leftmargin=*]
\item Any smooth one-dimensional $K$-variety with a $K$-point has infinitely many $K$-points.
\item If $f\in K[x,y]$ and $(a,b)\in K^2$ satisfy $f(a,b)=0\ne\der f/\der y (a,b)$, then $f$ has infinitely many zeros in $K^2$.
\end{enumerate}

This notion was introduced by Pop in the context of inverse Galois theory~\cite{pop-embedding} and independently by Sander in his unpublished thesis~\cite{Sander-thesis}. It was soon realized that this is the ``right class'' of fields where ``one can do a lot of [other] interesting mathematics''~\cite{Pop-little}.
Largeness can be seen as the opposite of satisfying some form of the Mordell Conjecture.
In particular, number fields and function fields are not large.
On the other hand, separably closed fields, real closed fields, and fields admitting non-trivial henselian valuations are large.
Fields with pro-$p$ absolute Galois group~\cite{Jarden-pro-p}, pseudo-algebraically closed ($\pac$) fields (such as pseudofinite fields and infinite algebraic extension of finite fields), pseudo-real closed ($\prc$) fields (such as the field of totally real algebraic numbers),  and more generally fields which satisfy a local-global principle with respect to a set of places~\cite[1.B]{Pop-little} are also large.
Algebraic extensions of large fields are large.
A longstanding open question asks whether the maximal abelian extension of $\Qq$ is large.
Much more can be said about this fascinating class of fields; we refer the reader to \cite{Pop-little, open-problems-ample} for excellent surveys.

\medskip
For a while, it was ``speculated''~\cite[Remark~1.5]{Pop-little} and ``generally believed''~\cite{Pophenselian} that fraction fields of complete noetherian local domains of Krull dimension $\ge 2$ (e.g., $K[[x,y]]$) were not large.
Efroymson and Pop independently refuted this, showing that the fraction field of a henselian local domain is large~\cite{efroymson,Pophenselian}.
We show here that all large fields can essentially be obtained this way:

\begin{thmA}
A field is large if and only if it is elementarily equivalent to the fraction field of a henselian local domain that is not a field.
\end{thmA}

Recall that two fields are elementarily equivalent if they satisfy the same first-order sentences in the language of rings.
Note that the right-to-left direction of Theorem~A follows from Efromyson and Pop's result as the class of large fields is elementary.

\medskip
A subfield $F$ of $K$ is existentially closed in $K$ if whenever $\alpha_1,\ldots,\alpha_n \in F$ and $\varphi(x_1,\ldots,x_n)$ is an existential formula in the language of rings  then $F$ satisfies $\varphi(\alpha_1,\ldots,\alpha_n)$ if and only if $K$ satisfies $\varphi(\alpha_1,\ldots,\alpha_n)$.
F.-V. Kuhlmann~\cite{kuhlmann} showed that $K$ is large if and only if $K$ is existentially closed in $K(\!(t)\!)$.
It is easy to see that an existentially closed subfield of a large field is large.
It therefore follows from Efroymson, Pop, and Kuhlmann's results that $K$ is large if and only if $K$ is an existentially closed subfield of the fraction field of a henselian local domain.
Thus Theorem~A is equivalent to showing that any existentially closed subfield of the fraction field of a henselian local domain is an elementary submodel of the fraction field of another henselian local domain.

\medskip
By Theorem~A any pseudofinite field is elementarily equivalent to the fraction field of a non-trivial henselian local domain.
This is quite counterintuitive to us.
It would be nice to have a model-theoretically tame henselian local domain with pseudofinite fraction field.

\medskip
Theorem~A is proven in Section~\ref{section:l->h}.
The proof makes essential use of the \'etale-open topology, introduced in~\cite{firstpaper}.
Let $V$ range over $K$-varieties.
The \textbf{\'etale-open topology} (or $\Sa E_K$-topology) on $V(K)$ is the topology with open basis the collection of \textbf{E-sets}, i.e., sets of the form $f(W(K))$ for a $K$-variety $W$ and an \'etale morphism $f\colon W\to V$.
The relevance to the setting of large fields follows from a result in~\cite{firstpaper}: $K$ is large if and only if the \'etale-open topology on $K = \Aa^1(K)$ is not discrete. Our second main result reads:

\begin{thmB}
Suppose that $K$ is not separably closed and $V \to W$ is an \'etale morphism of $K$-varieties.
Then the induced map $V(K)\to W(K)$ is a local homeomorphism in the \'etale-open topology.
Hence, if $V$ is smooth and irreducible then $V(K)$ is locally homeomorphic to $K^d = \Aa^d(K)$ for $d$ the dimension of $V$.
\end{thmB}

If $K$ is separably closed then the \'etale-open topology on $V(K)$ agrees with the Zariski topology for any $K$-variety $V$ and hence the conclusions of Theorem~B fail.

\medskip
Theorem~B is used to study Nash functions over large fields in Section~\ref{section:nashh-basic}, which are then used to prove Theorem~A in Section~\ref{section:l->h}.
The proof of Theorem~B in Section~\ref{sec: localhomeo} uses a result comparing the $\cE_K$-topology to another topology, which we now describe.
The \textbf{finite-closed topology} (or $\cF_K$-topology) on $V(K)$ is the topology with closed basis the collection of \textbf{F-sets}, i.e., sets of the form $f(W(K))$ for a $K$-variety $W$ and a finite morphism $f\colon W \to V$.

\medskip
The $\cF_K$-topology can also be described in a lengthier, but somewhat more natural, fashion.
Let $V\to W$ be a morphism of $K$-varieties, let $V(K)\to W(K)$ be the induced map, and equip $V(K)$ and $W(K)$ with the $\cF_K$-topology.
In Section~\ref{section:FCT}, we show the following:
\begin{enumerate}[leftmargin=*]
\item $V(K)\to W(K)$ is continuous.
\item If $V\to W$ is a scheme-theoretic open \textup{(}closed\textup{)} immersion then $V(K)\to W(K)$ is a topological open \textup{(}closed\textup{)} immersion.
\item If $V\to W$ is finite then $V(K)\to W(K)$ is a closed map.
\end{enumerate}
Here (1) and (2) show that the finite-closed topology is a ``system of topologies over $K$" in the sense of \cite{firstpaper} and (3) shows that the $\cF_K$-topology is the coarsest system of topologies in which finite morphisms induce closed maps.
This is analogous to how the $\cE_K$-topology is the coarsest system of topologies in which \'etale morphisms induce open maps~\cite{firstpaper}.

\medskip
We now discuss our third main theorem, the promised comparison Theorem~C.
Recall that $K$ is \textbf{bounded} if $K$ has only finitely many separable extensions of any given degree $d < \infty$.
We say that one system of topologies $\cT_1$ refines another system $\cT_2$ if the $\cT_1$-topology on $V(K)$ refines the $\cT_2$-topology on $V(K)$ for every $K$-variety $V$.

\begin{thmC}
\hspace{.0001cm}
\begin{enumerate}[leftmargin=*]
\item If $K$ is perfect then the $\cE_K$-topology refines the $\cF_K$-topology.
\item If $K$ is  bounded then the $\cF_K$-topology refines the $\cE_K$-topology.
\item If $K$ is perfect and t-henselian then the $\cF_K$-topology agrees with the $\cE_K$-topology.
\end{enumerate}
\end{thmC}

Theorem~C.1 (Section~\ref{sec: etale-refine}) is the key ingredient used to prove Theorem~B above (Section~\ref{sec: localhomeo}) and to show that the \'etale-open topology on $K$ is almost always zero-dimensional (Section~\ref{section:?}). 

\medskip
Theorems~C.1 and~C.2 (Section~\ref{sec: boundedfields}) imply that the $\cF_K$-topology agrees with the $\cE_K$-topology when $K$ is a perfect and bounded field, also referred to as a field of type (F) by Serre~\cite[\S{}4.2]{Serre_gcoho}.
In particular, these results apply to algebraically, real, and $p$-adically closed fields.
Combining with the characterization of the  $\cE_K$-topology in these settings, the $\cE_K$- and $\cF_K$-topologies both agree with the Zariski, order, or valuation topologies, respectively.
One can also obtain this agreement directly using arguments involving quantifier elimination; see Section~\ref{section:rcr-pcf}.
This yields a stronger conclusion: when $K$ is algebraically, real, or $p$-adically closed, a subset of $K^m$ is an E-set (F-set) if and only if it is definable in the language of rings and open (closed) with respect to the Zariski, order, or valuation topology, respectively.

\medskip
There are many more examples of perfect bounded fields.
Pseudofinite fields and infinite algebraic extensions of finite fields are examples of perfect bounded fields which are also $\pac$.
If the absolute Galois group $\Gal(K)$ of $K$ is topologically finitely generated then $K$ is bounded~\cite[\S{}4.1]{Serre_gcoho}.
It is easy to see that $\Gal(K)$ is topologically finitely generated if and only if $K$ is the fixed field of finitely many elements of $\Gal(K)$.
The theory of a field equipped with $m$ field orders has a model companion for any $m \ge 1$, and if $(K,<_1,\ldots,<_m)$ satisfies this model companion, then $K$ is $\prc$ and $\Gal(K)$ is a pro-$2$-group topologically generated by $m$ involutions~\cite{prestel-prc}. 
Finally, suppose that $L$ is an algebraically closed field of characteristic zero and fix irrational $\beta \in L$.
If $K$ is a maximal subfield of $L$ not containing $\beta$, then $K$ is bounded by~\cite{Quigley}.

\medskip
There are also unbounded perfect t-henselian fields such as $\Qq(\!(t)\!)$, where Theorem~C.3 (Section~\ref{section:t-henselian}) allows us to show the two topologies agree. 

\medskip
Theorem~C.1 requires $K$ to be perfect.
If $K$ is separably closed of characteristic $p$ but not algebraically closed then the set of $p$th powers is $\cF_K$-closed but not $\cE_K$-closed.
Even worse, Theorem~D shows that the finite-closed topology on a bounded large field can be discrete.

\begin{thmD}
There is a bounded $\pac$ field $K$ such that:
\begin{enumerate}[leftmargin=*]
\item The $\cF_K$-topology is discrete and the $\cE_K$-topology is non-discrete.
\item There is a quintic polynomial map $K \to K$ whose image is $K^\times$.
\end{enumerate}
\end{thmD}

See Section~\ref{section:Esharp} for the proof of Theorem~D.
In 2009 Philipp Lampe asked on MathOverflow if there is an infinite field $K$ and polynomial $h \in K[x]$ such that $K \setminus h(K)$ is nonempty and finite~\cite{6820}.
Theorem~D answers this question.
Previously, Kosters~\cite{kosters} had shown that if $K$ is perfect and large and $K \setminus h(K)$ is nonempty then $|K \setminus h(K)| = |K|$.
Kosters' result is an easy corollary of Theorem~C; see Section~\ref{section:Esharp}.

\medskip
We could not determine whether there is an infinite \emph{perfect} field $K$ on which the $\cF_K$-topology is discrete.
This question is related to important open problems in field theory; see Section~\ref{sec:open}.

\medskip
We remark that the order of discovery was amusingly in reverse: We defined the finite-closed topology before the \'etale-open topology.
To compare these for fields without obvious topologies, such as pseudofinite fields, we proved Theorem~C.  We then realized it can be used to prove Theorem~B.  Several years later, we proved Theorem~A after an unsuccessful attempt by the first author to \emph{disprove} Theorem~A for pseudofinite fields.
In between, we connected the \'etale-open topology with large fields and proved the results in~\cite{firstpaper}.
Subsets of the authors and others also obtained related results \cite{secondpaper,field-top-1,field-top-2,with-anand, topological_proofs}.

\subsection*{A note on the title}\label{section:title note}
The title of this paper is incorrect on the level of fields.
There are fields which are large but not fraction fields of henselian local domains.
One example is the field of real numbers, see Fact~\ref{fact:sharp}.
However, our title is mathematically correct on the level of theories of fields.
Theorem~A shows that the theory of large fields is exactly the theory of fraction fields of henselian local domains and that the theory of any particular large field is the theory of a fraction field of a henselian local domain.

\subsection*{Acknowledgments}
Will Johnson was funded in part by the US National Science Foundation (Grant No.\@ DMS-1803120), the National Natural Science Foundation
  of China (Grants No.\@ 12101131 and No.\@ W2532009) and the Ministry of Education of
  China (Grant No.\@ 22JJD110002).
  Erik Walsberg was funded in part by the Austrian Science Fund (FWF) 10.55776/PAT1673125.

\section{Conventions and background}\label{section:c&b}

We will adopt the conventions in this section for the rest of the paper.
We let $d, k, m, n$ be natural numbers.

\subsection{Basic algebra and algebraic geometry} \label{sec: basicalgngeo}
All rings are commutative with unit.
Throughout, $K$ is a field.
We let $\kalg$ be the algebraic closure of $K$.
A $K$-algebra is a ring $R$ equipped with a homomorphism $K \to R$.
We let $\Frac(R)$ be the fraction field of a domain $R$.
A local ring $R$ with maximal ideal $\mm$ is \textbf{henselian} if any simple root of some $f \in R[x]$ in $R/\mm$ lifts to a root of $f$ in $R$.  
This is a form of Hensel's lemma.  
See \cite[Lemma~04GG]{stacks-project} for other characterizations of henselianity.
We say that a local domain is \textbf{proper} if it is not a field.

\medskip
A \textbf{$K$-variety} is a separated reduced scheme of finite type over $K$.
Throughout, $V,W,V',W',$ etc.\@ range over $K$-varieties and a morphism $V \to W$ is always a morphism of $K$-varieties.
We let $V(K)$ be the set of $K$-points of $V$.
Given a morphism $V \to W$ we let $V(K) \to W(K)$ be the induced map on $K$-points.
We let $\Aa^n$ be $n$-dimensional affine space over $K$ for $n\ge 1$, i.e., $\Aa^n = \Spec K[x_1,\ldots,x_n]$.
Recall that $\Aa^n(K)$ is canonically identified with $K^n$ and so $\Aa^1(K)$ is canonically identified with $K$.
Given a field extension $L/K$ we let $V_L = V \times_{\Spec K} \Spec L$ be the base change of $V$ to an $L$-variety.
Given a morphism $f\colon V \to W$ we let $f_L\colon V_L\to W_L$ be the base change of $f$.
Recall that $V_L(L)$ is canonically identified with $V(L)$.

\medskip
We let $A_b$ be the localization of a ring $A$ at $b \in A$.
Given $f,g \in A[x]$ with $f$ monic, the map $A \to A[x]_g/(f)$ is  standard \'etale if the derivative $f'$ of $f$ is invertible in $A[x]_g/(f)$. 
A morphism $X\to S$ of affine schemes is {\bf standard \'etale} if, up to isomorphism, $X \to S$ is dual to a standard \'etale ring morphism. The more appropriate definition of an {\bf \'etale} morphism can be found at~\cite[Section~02GH]{stacks-project}.  Fact~\ref{fac: sdetalevsetale} establishes the relationship between \'etale and standard \'etale morphisms, which can be taken as the working definition of \'etale morphism for our purposes.
It is~\cite[Lemma~02GT]{stacks-project}.

\begin{fact} \label{fac: sdetalevsetale}
Let $f\colon X \to S$ be an \'etale morphism of schemes.
\begin{enumerate}[leftmargin=*]
\item There are covers $\{X_i\}_{i \in I}$ of $X$ and $\{S_i\}_{i \in I}$ of $S$ by affine open subschemes such that $f(X_i) \subseteq S_i$ and $X_i \to S_i$ is standard \'etale for every $i \in I$.
\textup{(}Conversely, if such a cover exists, then $f$ is \'etale.\textup{)}
\item If $S$ is affine then there is a cover $\{X_i\}_{i \in I}$ of $X$ by affine open subschemes such that $X_i \to S$ is standard \'etale for every $i$ in $I$.
\end{enumerate}
When $X$ and $S$ are quasi-compact, we can take the covers to be finite covers.
\end{fact}

We record some well-known facts about \'etale morphism.

\begin{fact}
\label{fact:etale}
\quad
\begin{enumerate}[leftmargin=*]
\item \'Etale morphisms are closed under (arbitrary) base change and composition.
\item  Open immersions are \'etale.
\item If $f \colon V \to W$ is an \'etale morphism of $K$-varieties then the image of $f$ is an open subvariety of $V$.
\end{enumerate}
\end{fact}

\begin{proof}
These can be found in \cite[Prop.~17.1.3]{EGA-IV-4} or 
alternatively~\cite[Lemmas  02GN, 02G0, 02GP,  03WT]{stacks-project} .
\end{proof}

A morphism $\Spec B \to \Spec A$ of affine schemes is finite if the dual ring morphism $A \to B$ gives $B$ the structure of a finitely generated $A$-module.
A morphism $f\colon S \to X$ of schemes is {\bf finite} if the following equivalent conditions hold:
\begin{enumerate}[leftmargin=*]
\item Every $p \in X$ is contained in an affine open subscheme $X'$ of $X$ such that $f^{-1}(X')$ is also affine and the induced morphism $f^{-1}(X') \to X'$ is a finite morphism of affine schemes.
\item For every affine open subscheme $X' \subseteq X$, the preimage $f^{-1}(X')$ is also affine and the induced morphism $f^{-1}(X') \to X'$ is a finite morphism of affine schemes.
\end{enumerate}
See \cite[Lemma~01WI]{stacks-project} for a proof of the equivalence.
Recall that a morphism of $K$-varieties is {\bf quasi-finite} if it has finite fibers.

\begin{fact}\label{fact:basic-f}
\hspace{.000001cm}
\begin{enumerate}[leftmargin=*]
\item Finite morphisms are closed under (arbitrary) base changes and compositions.
\item\label{afp} A morphism is finite if and only if it is quasi-finite and proper if and only if it is affine and proper.
\item If $V_i \to W$ is a finite morphism  for $i=1,\ldots,n$ and $V$ is the disjoint union of the $V_i$, then the induced morphism $V \to W$ is finite.
\item Closed immersions are finite.
\item\label{poly:fini} Any non-constant morphism $\Aa^1 \to \Aa^1$ is finite.
\end{enumerate}
\end{fact}

Here, (1) is \cite[Lemmas~01WK, 01WL]{stacks-project}, (2) is \cite[Lemmas~01WN, 02OG]{stacks-project}, (3) is \cite[Lemma~0CYI]{stacks-project} and (4) is~\cite[Lemma~035C]{stacks-project}.
For (5), note that a non-constant morphism $\Aa^1 \to \Aa^1$ is dual to the $K$-algebra morphism $K[t] \to K[t], t \mapsto f$ for some non-constant $f \in K[t]$.
Such a ring morphism makes $K[t]$ into a finite module over itself.
We now recall a form of Zariski's main theorem.

\begin{fact}\label{fact:zmt}
Any quasi-finite morphism $V \to W$ factors as $V \to V^* \to W$ where $V \to V^*$ is an open immersion with Zariski dense image and $V^* \to W$ is finite.
\end{fact}

\begin{proof}
By \cite[Lemma~05K0]{stacks-project} we can factor $V \to W$ as $V \to V^{**} \to W$ where $V \to V^{**}$ is an open immersion and $V^{**} \to W$ is finite.
Let $V^{*}$ be the Zariski closure of the image of $V \to V^{**}$.
The inclusion $V = V \cap V^* \hookrightarrow V^*$ is an open immersion, because pullbacks of open immersions are open immersions.
Hence $V \to V^{*}$ is an open immersion with Zariski dense image.
Furthermore the inclusion $V^{*} \to V^{**}$ is a closed immersion and is hence finite and so $V^{*} \to W$ is a composition of finite morphisms and is hence finite.
\end{proof}

\subsection{Systems of topologies and the \'etale-open topology} \label{sec: systemoftop}

\begin{definition}\label{sys-def}
A \textbf{system of topologies} over $K$ is a choice of a topology on $V(K)$ for every  $V$, such that the following conditions hold:
\begin{enumerate}[leftmargin=*]
\item $V(K)\to W(K)$ is continuous for any morphism $V \to W$.
\item If $V \to W$ is a (scheme-theoretic) open immersion, then
$V(K) \to W(K)$ is a (topological) open embedding.
\item If $V \to W$ is a (scheme-theoretic) closed immersion, then
$V(K) \to W(K)$ is a (topological) closed embedding.
\end{enumerate}
\end{definition}

Given systems $\cT,\cT^*$ of topologies over $K$ we say that $\cT$ \textbf{refines} $\cT^*$ if the $\cT$-topology on $V(K)$ refines the $\cT^*$-topology on $V(K)$ for every $V$.
The finest system of topologies over $K$ is the discrete system in which every $V(K)$ has the discrete topology and the coarsest system of topologies over $K$ is the Zariski system in which every $V(K)$ has the Zariski topology.

\begin{fact}\label{fact:prod}
Suppose that $\cT$ is a system of topologies over $K$ and $V_1,V_2$ are $K$-varieties.
Then the $\cT$-topology on $V_1(K) \times V_2(K) = (V_1 \times V_2)(K)$ refines the product of the $\cT$-topologies on $V_1(K)$ and $V_2(K)$.
\end{fact}

Fact~\ref{fact:prod} follows as each projection $V_1(K) \times V_2(K) \to V_i(K)$ is $\cT$-continuous.

\begin{definition}
The \textbf{$K$-rational image} of a morphism $V\to W$ is the image of  the induced map $V(K)\to W(K)$.
\end{definition}

We recall the \'etale open topology introduced in~\cite{firstpaper}.

\begin{definition}\label{def:E}
An {\bf E-subset} of $V(K)$ is a $K$-rational image of an \'etale morphism $W \to V$.
By Fact~\ref{fact:E-set} below E-sets form an open basis for a topology on $V(K)$.
We call this the {\bf \'etale-open} topology on $V(K)$.
We let $\cE_K$ denote the collection of \'etale-open topologies on $V(K)$ for $V$ ranging over $K$-varieties.
The \'etale-open topology over $K$ is the coarsest system of topologies over $K$ such that $V(K) \to W(K)$ is an open map  for any \'etale morphism $V \to W$ \cite[Prop.~5.5]{firstpaper}.
\end{definition}

Fact~\ref{fact:E-set} below is \cite[Lemma~3.3]{firstpaper}.

\begin{fact}\label{fact:E-set}
E-subsets of $V(K)$ are closed under finite unions and intersections.
\end{fact}

We recall some facts we will use about the \'etale-open topology.

\begin{fact}[\cite{firstpaper,secondpaper,field-top-1,field-top-2}]\label{fact:old-EO}
~
\begin{enumerate}[leftmargin=*] 
\item\label{old:lrge} $K$ is large if and only if the $\cE_K$-topology on $K$ is not discrete.
\item\label{old:sep} If $K$ is separably closed, then the $\cE_K$-topology agrees with the Zariski topology.
Otherwise, the $\cE_K$-topology on $V(K)$ is Hausdorff when $V$ is quasi-projective.
\item\label{old:ordr} The $\cE_K$-topology refines the topology induced by any field order on $K$.
If $K$ is real closed then the $\cE_K$-topology agrees with the topology induced by the order.
\item\label{old:tsepr} If $K$ is neither separably closed nor isomorphic to $\Rr$ then the $\cE_K$-topology on $K$ is totally separated and hence totally disconnected.
\item\label{old:hnsl} If $K$ is t-henselian and not separably closed then the $\cE_K$-topology agrees with the canonical t-henselian topology.
In particular if $K$ is not separably closed and admits a non-trivial henselian valuation then the $\cE_K$-topology agrees with the valuation topology.
\item\label{old:dvr} If $v$ is a valuation on $K$ and either the value group of $v$ is not divisible or the residue field of $v$ is not algebraically closed then the $\cE_K$-topology refines the $v$-topology.
\item\label{old:pacv2} If $K$ is $\pac$ then the $\cE_K$-topology is not induced by a field topology on $K$.
\end{enumerate}
\end{fact}
See \cite{field-top-1, field-top-2} for some additional results on the \'etale-open topology in other cases.
Fact~\ref{fact:z-dense} below is \cite[Props.~5, 6]{topological_proofs}.

\begin{fact}\label{fact:z-dense}
Suppose that $K$ is large, $V$ is irreducible, and $O$ is an $\cE_K$-open subset of $V(K)$ which contains a smooth point.
Then $O$ is Zariski dense in $V$ and $|O| = |K|$.
\end{fact}

\subsection{Algebraic power series}
Fix an index set $I$. \label{sec:algpower}
Let $K[t_i]_{i \in I}$ be the polynomial ring over $K$ with variables indexed by $I$.
Let $\pringI$ be the ring of power series over $K$ with variables indexed by $I$, where each series depends on only finitely many variables.
This is the colimit of the rings $K[[t_{i_1},\ldots,t_{i_m}]]$ with $i_1,\ldots,i_m \in I$.
Hence $\pringI$ is a colimit of henselian local domains along local embeddings and so $\pringI$ is a henselian local domain.

\medskip
Let $\palgI$ be the relative algebraic closure of $K[t_i]_{i\in I}$ in $\pringI$, i.e., the ring of algebraic power series over $K$ with variables indexed by $I$.
When $I = \{1,\ldots,m\}$ we write $\palg$.
Again, $\palgI$ is the colimit of the rings $K[t_{i_1},\ldots,t_{i_m}]^\mathrm{h}$ for $i_1,\ldots,i_m\in I$.

\begin{fact}\label{fact:palg}
Fix an index set $I$.
Then $\palgI$ is the henselization of the localization of $K[t_i]_{i \in I}$ at the ideal generated by the $t_i$.
\end{fact}

\begin{proof}
It suffices to treat the case when $I = \{1,\ldots,m\}$ as henselization commutes with limits.
Let $A$ be the henselization of $K[t_1,\ldots,t_m]_{(t_1,\ldots,t_m)}$.
Then $A$ is contained in $\palg$ as $A$ is an algebraic extension of $K[t_1,\ldots,t_m]$ contained in $\pring$, and $A$ is algebraically closed in $\pring$ by \cite[Cor.~44.3]{nagata-local}.
\end{proof}

Alternatively, the reader can \emph{define} $\palg$ as the henselization in Fact~\ref{fact:palg}.
We will never need the fact that $\palg$ is the ring of algebraic power series.

\subsection{Model theory} \label{sec:basicmodel}
We assume some familiarity with elementary notions of first order model theory.
We always consider fields to be first order structures in the language of rings.
Let $\cM,\cN$ be first order structures.
We write $\cM\equiv \cN$ when $\cM$ is elementarily equivalent to $\cN$ and $\cM \preceq \cN$ when $\cM$ is an elementary submodel of $\cN$.
Given $A \subseteq \cM$ we say that a subset $X$ of $\cM^m$ is $A$-definable if it is definable in $\cM$ with parameters from $A$, and we simply say that $X$ is definable if it is definable with parameters from $\cM$.
Let $\kappa$ range over infinite cardinals.
We let $\kappa^+$ be the cardinal successor of $\kappa$.
Recall that if $\kappa$ is uncountable then a structure $\cM$ in a countable language (e.g., a ring) is {\bf $\kappa$-saturated} if any family of at most $\kappa$ definable sets with the finite intersection property has non-empty intersection.
By \cite[Cor.~10.2.2]{Hodges} a structure in a countable language has a $\kappa$-saturated elementary extension for any $\kappa$.

\medskip
Fact~\ref{fact:lw} is a special case of the downwards \lowenheim theorem~\cite[Cor.~3.1.5]{Hodges}.

\begin{fact}\label{fact:lw}
Let $\Sa M$ be a first-order structure in a countable language.
Any subset of $\Sa M$ of cardinality $\kappa$ is contained in an elementary submodel of $\Sa M$ of cardinality $\kappa + \aleph_0$.
\end{fact}

We often use Fact~\ref{fact:lw} without mention.

\section{The finite-closed topology}\label{section:FCT}

We introduce the finite-closed topology and obtain a characterization similar to the characterization of the \'etale-open topology at the end of Definition~\ref{def:E}.

\begin{definition}
An \textbf{F-subset} of $V(K)$ is a  $K$-rational image of a finite morphism $W \to V$.
When $V(K)$ is clear from context, we will refer to this as an \textbf{F-set}.
\end{definition}

The collection of F-sets is closed under a number of operations:

\begin{lemma}\label{fc-lemma}
Let $f \colon V \to W$ be a morphism.
\begin{enumerate}[leftmargin=*]
\item The preimage of an F-set under $V(K) \to W(K)$ is an F-set.
\item If $V \to W$ is finite then the image of an F-set under $V(K) \to W(K)$ is an F-set.
\item F-subsets of $V(K)$ are closed under finite intersections and
finite unions.
\end{enumerate}
\end{lemma}

\begin{proof}
(2) holds as finite morphisms are closed under composition.
We prove (1).
Fix an F-subset $X$ of $V(K)$.
Let $\phi \colon V' \to V$ be a finite morphism such that $\phi(V'(K)) = X$.
The pullback
$\phi^* \colon V' \times_V W \to W$ of $\phi$ is  finite by Fact~\ref{fact:basic-f}, and the pullback square
\[ \xymatrix{ (V' \times_V W)(K) \ar[r] \ar[d]_{\phi^*} & V'(K) \ar[d]^\phi \\ W(K) \ar[r]_f & V(K)}\]
shows that the $K$-rational image of $\phi^*$ agrees with the preimage of $X$ under $f$.

\medskip
Let $X_1,X_2$ be F-subsets of $V(K)$.
Fix finite $f_i\colon W_i\to V$ with $f_i(W_i(K))=X_i$ for $i = 1,2$.
Let $f_\times \colon W_1\times_V W_2 \to V$ and $f_\sqcup\colon W_1\sqcup W_2 \to V$ be the natural maps.
By Fact~\ref{fact:basic-f} $f_\times, f_\sqcup$ are finite.
The $K$-rational images of $f_\times$, $f_\sqcup$ are $X_1\cap X_2$, $X_1\cup X_2$, respectively.
\end{proof}

Lemma~\ref{fc-lemma}(3) shows that F-subsets of $V(K)$ form a closed basis for a topology on $V(K)$.

\begin{definition}
The \textbf{finite-closed topology} on $V(K)$ is the topology with closed basis the collection of F-subsets of $V(K)$.
We let $\cF_K$ denote the collection of finite-closed topologies on $V(K)$ with $V$ ranging over $K$-varieties.
We call $\cF_K$ the finite-closed topology over $K$.
\end{definition}

\medskip
Below is the main result of this section.

\begin{theorem} \label{215}
The finite-closed topology is a system of topologies over $K$, and can be characterized as the coarsest system of topologies such that $V(K) \to W(K)$ is a closed map for every finite morphism $V \to W$.
\end{theorem}

Theorem~\ref{215} follows from Lemmas~\ref{lem:system-1}, \ref{lem:system-2}, and \ref{lem:system-4} in the remaining part of the section.

\medskip
Lemma~\ref{lem:system-1} below verifies  (1) in Definition~\ref{sys-def} of a system of topologies. It follows immediately from Lemma~\ref{fc-lemma}(1).

\begin{lemma}\label{lem:system-1}
If $V \to W$ is a morphism, then  $V(K) \to W(K)$ is continuous with respect to the finite-closed topology.
\end{lemma}

We recall a criterion for a map to be closed.

\begin{lemma}\label{closed-lemma}
Let $f \colon X \to Y$ be a continuous map of topological spaces and $\mathcal{B}$ be a closed basis for $Y$.
Suppose that:
\begin{itemize}
\item $f$ has finite fibers.
\item $f(C_1 \cap \cdots \cap C_n)$ is closed for any $C_1, \ldots, C_n\in\mathcal{B}$.
\end{itemize}
Then $f$ is a closed map.
\end{lemma}

\begin{proof}
Let $C$ be a closed subset of $X$.
We show that $f(C)$ is closed.
Fix $b \in Y\setminus f(C)$.
We show that $b$ is not in the closure of $f(C)$.
Let $f^{-1}(b) = \{a_1,\ldots,a_n\}$.
Then $a_i\notin C$ for $i=1,\ldots,n$, so for each $i$ there
is $C_i\in\mathcal{B}$ such that $C \subseteq C_i$ and $a_i
\notin C_i$.
Let $C^* = C_1 \cap \cdots \cap C_n$ and note that $f(C^*)$ contains $f(C)$.
By assumption,
$f(C^*)$ is closed.
By choice of the $C_i$, no $a_i$ is in $C^*$.
Hence $b \notin f(C^*)$, so $b$ is not in the closure of $f(C)$.
\end{proof}

The second part of Lemma~\ref{lem:system-2} below verifies (3) of Definition~\ref{sys-def}. In any system of topologies where finite morphisms induce closed maps, the F-sets must be closed. Thus, assuming we have shown the finite-closed topology is a system of topologies, the first part of Lemma~\ref{lem:system-2} will show the remaining part of Theorem~\ref{215}.

\begin{lemma}\label{lem:system-2}
Let $f \colon V \to W$ be a finite morphism.
Then $V(K) \to W(K)$ is a closed map with respect to the finite-closed topology.
Hence if $V \to W$ is a closed immersion, then $V(K) \to
W(K)$ is a closed embedding in the $\cF_K$-topology.
\end{lemma}

\begin{proof}
The second claim follows from the first as closed immersions are finite, and injective continuous closed maps are closed embeddings.
We prove the first claim.
Finite maps have finite fibers by Fact~\ref{fact:basic-f}.
We apply Lemma~\ref{closed-lemma}.
Let $C_1, \ldots, C_n$ be F-subsets of $V(K)$. 
It suffices to show that $f(C_1 \cap \cdots \cap
C_n)$ is $\cF_K$-closed in $W(K)$.
In fact, $f(C_1 \cap \cdots \cap C_n)$ is an F-set by parts (2) and (3) of Lemma~\ref{fc-lemma}.
\end{proof}

Fact~\ref{enough-of-this} is essentially trivial from the definitions, using in part (2) the analogous statement about rings\footnote{If $A \to B \to C$ are ring homomorphisms such that $B$ and $C$ are finitely generated $A$-modules, then $C$ is a finitely generated $B$-module.
Indeed, any set which generates $C$ as an $A$-module also generates it as a $B$-module.}.
Alternatively, (2) can be proven from \cite[Lemma~035D]{stacks-project}.

\begin{fact} \label{enough-of-this} 
Let $f \colon W \to W'$ and $g \colon W' \to W''$ be morphisms.
\begin{enumerate}
\item If $g \circ f$ and $g$ are open immersions, then $f$ is an open immersion.
\item If $g \circ f$ and $g$ are finite morphisms, then $f$ is finite.
\end{enumerate}
\end{fact}

The proof of Lemma~\ref{open-lemma} below is an elementary topological exercise which we leave to the reader.

\begin{lemma}\label{open-lemma}
Let $X,Y$ be topological spaces, $\mathcal{B}$ be a closed basis for $X$, and $f \colon X \to Y$ be continuous.
Suppose that we have the following.
\begin{itemize}
\item $f$ is injective
\item $f(X)$ is open in $Y$.
\item For every  $C \in\mathcal{B}$, there is a closed
set $C^* \subseteq Y$ such that $C = f^{-1}(C^*)$.
\end{itemize}
Then $f$ is an open embedding.
\end{lemma}

Lemma~\ref{extend-lemma} is the main fact from algebraic geometry used in Lemma~\ref{lem:system-4} below.

\begin{lemma}\label{extend-lemma}
Let $V$ be an open subvariety of a $K$-variety $V^*$ and $\phi\colon W \to V$ be a finite morphism.
Then there is a finite morphism $\phi^* \colon W^* \to V^*$ whose pullback to $V$ is $\phi$.
\end{lemma}

\begin{proof}
Let $W\to V^*$ be the composition of $W\to V$ with $V\to V^*$.
Then $W \to V^*$ has finite fibers and is hence a quasi-finite morphism.
Applying Zariski's main theorem (Fact~\ref{fact:zmt}), we factor $W\to V^*$ as $W \to W^* \to V^*$ where $W \to W^*$ is an open embedding with Zariski dense image and $W^* \to V^*$ is a finite morphism.
\[\xymatrix{W \ar[r] \ar[d] & W^* \ar[d] \\ V \ar[r] & V^*}\]
Let $X$ be the pullback $V \times_{V^*} W^*$.
We get the following diagram.
\begin{equation*}
\xymatrix{W \ar[drr]^{\mathrm{open}} \ar[dr] \ar[ddr]_{\mathrm{finite}} & & \\ & X \ar[r]_{\mathrm{open}} \ar[d]^{\mathrm{finite}} & W^* \ar[d]^{\mathrm{finite}} \\ & V \ar[r]_{\mathrm{open}} & V^*}
\end{equation*}
Here the arrows marked with ``open'' are open immersions, and the arrows marked with ``finite'' are finite morphisms.
The morphism $W \to X$ must be a finite open immersion by Fact~\ref{enough-of-this}. 
Hence $W\to X$ is a clopen embedding as the image of a finite morphism is closed.
Let $D \to X$ be the complementary clopen embedding.
The images of $D$ and $W$ in $W^*$ are disjoint open subvarieties, but the image of $W$ is Zariski dense in $W^*$.
Therefore $D = \varnothing$, hence $W \to X$ is an isomorphism, and so $W$ is the pullback of $W^*$ along $V \to V^*$.
\end{proof}

Lemma~\ref{lem:system-4} below verifies (2) of Definition~\ref{sys-def} for the finite-closed topology.

\begin{lemma}\label{lem:system-4}
Suppose that $V\to W$ is an open immersion.
Then $V(K) \to W(K)$ is an open immersion in the $\cF_K$-topology.
\end{lemma}

\begin{proof}
Let $C \to W$ be the complementary closed subvariety.
The image of $C(K) \to W(K)$ is $\cF_K$-closed by Lemma~\ref{lem:system-2}, so the image of $V(K) \to W(K)$ is $\cF_K$-open.
By Lemma~\ref{lem:system-1} $V(K) \to W(K)$ is $\cF_K$-continuous.
Hence by Lemma~\ref{open-lemma} it suffices to show that every F-subset of $V(K)$ is the restriction to $V(K)$ of an F-subset of $W(K)$.
This follows immediately from Lemma~\ref{extend-lemma}.
\end{proof}

\section{The $\cE_K$-topology refines the $\cF_K$-topology when $K$ is perfect} \label{sec: etale-refine}

We prove Theorem C.1 from the introduction, while Theorems C.2 and C.3 will be proven in Section~\ref{sec: comparison}.
This reflects the fact that only Theorem C.1 will be used in Sections~\ref{section:?}--\ref{section:l->h}.

\medskip
We let $R^{\red}$ be the quotient of a ring $R$ by its nilradical.
Lemma~\ref{lem:henselian-finite-section} is the commutative algebraic result behind Theorem~\ref{thm:etale-refine} below ($=$ Theorem C.1).

\begin{lemma} \label{lem:henselian-finite-section}
Suppose that $K$ is perfect.
Let $A$ be a noetherian henselian local $K$-algebra with residue field $K$ and maximal ideal $\mm$.
Let $B$ be a finite $A$-algebra.
Then the quotient homomorphism $B \to (B/\mm B)^{\red}$ has a section $(B/\mm B)^{\red} \to B$ in the category of $K$-algebras.
\end{lemma}

\begin{proof}
By~\cite[Lemma~04GH]{stacks-project}, $B$ is a finite product of henselian local rings $B_i$ such that each $B_i$ is a finite $A$-algebra.
We can thus assume that $B$ is a henselian local ring and that $B$ is finite over $A$.
By Nakayama's lemma, $B \ne \mm B$, so  $B/\mm B$ is non-trivial.
Note that $B/\mm B = B \otimes_A (A/\mm)$ is a finite algebra over $A/\mm = K$, hence artinian.
As an artinian local ring, $B/ \mm B$ has a unique prime ideal, necessarily induced by the maximal ideal $\mm^*$ of $B$.
It follows that $(B/ \mm B)^{\red}=B/ \mm^*$.
Moreover, $B/\mm^*$ is a finite field extension of $K$.

\medskip
By the primitive element theorem we have $B/\mm^* = K(a)$ for some $a\in B/\mm^*$.
Let $h$ be the minimal polynomial of $a$ over $K$.
Then $h$ is separable because $K$ is perfect.
We can identify $h$ as a polynomial over $B$, as $B$ is a $K$-algebra.
By Hensel's lemma, $a$ can be lifted to a root of $h$ in $B$.
This gives the desired homomorphism $B/\mm^* \to B$.
\end{proof}

Lemma~\ref{main-lemma} essentially recasts Lemma~\ref{lem:henselian-finite-section} in geometric language.

\begin{lemma} \label{main-lemma}
Suppose that $K$ is perfect.
Let $W \to V$ be a finite morphism of affine $K$-varieties, $p \colon \Spec K \to V$ be a $K$-point of $V$, and $F$ be the scheme-theoretic fiber of $W \to V$ over $p$, so that
\[ \xymatrix{ F \ar[r] \ar[d] & W \ar[d] \\ \Spec K \ar[r]_-p & V}\]
is a pullback square in the category of schemes.
Let $F^{\red}$ be the induced reduced subscheme, so that
\begin{equation}
\xymatrix{ F^{\red} \ar[r] \ar[d] & W \ar[d] \\ \Spec K \ar[r]_-p & V} \tag{$\star$}
\end{equation}
is a pullback square in the category of varieties.
Then we can extend to a diagram of schemes
\[ \xymatrix{ F^{\red} \ar[r] \ar[d] & W' \ar[d] \ar[r] & W \ar[d] \\ \Spec K \ar[r] & V' \ar[r] & V}\]
where
\begin{enumerate}[leftmargin=*]
\item $V' \to V$ is etale.
\item In the category of schemes, the right-hand square is a pullback square; equivalently, we have $W' = W \times_V V'$.
\item The outer rectangle is \textup{(}$\star$\textup{)}.
\item The map $F^{\red} \to W'$ has a retract, i.e., a morphism $W' \to F^{\red}$ of $K$-schemes 
such that the composition
\[ F^{\red} \to W' \to F^{\red}\]
is the identity.
\end{enumerate}
\end{lemma}

\begin{proof}
Let $V=\Spec A$, $W=\Spec B$, and $A \to B$ be the morphism dual to $W \to V$, so $A \to B$ is finite.
Take a maximal ideal $\mathfrak{m}\subseteq  A$ such that $A\to A/\mathfrak{m}$ is dual to $p\colon \Spec K\to A$.
We have the following pushout diagram in the category of reduced $K$-algebras.
\[
\begin{tikzcd}
A \arrow[d] \arrow[r] & A/\mathfrak{m} \arrow[d]    \\
B \arrow[r]           & (B\otimes_A A/\mathfrak{m})^{\red} 
\end{tikzcd}
\]
By an \textbf{\'etale neighborhood} of $\mm$ we mean an \'etale $A$-algebra $A \to A'$ with a morphism $A' \to A/\mm$ extending $A \to A/\mm$. 
The \textbf{henselization} $A_\mm^\mathrm{h}$ is the filtered colimit $A_\mm^\mathrm{h} = \varinjlim_{A'} A'$ where $A'$ ranges over \'etale neighborhoods of $\mm$ \cite[Definition~04GQ, Lemma~04GV]{stacks-project}.
The ring $A_\mm^\mathrm{h}$ is a henselian local ring with maximal ideal $\mm A_\mm^\mathrm{h}$, and its residue field is $A_\mm/\mm A_\mm = A/\mm$ \cite[Lemma~04GN]{stacks-project}.
Moreover, $A_\mm^\mathrm{h}$ is noetherian \cite[Lemma~06LJ]{stacks-project}.
We have the following diagram in the category of $K$-algebras:
\[
\begin{tikzcd}
A \arrow[d] \arrow[r] & A_\mathfrak{m}^\mathrm{h} \arrow[d] \arrow[r] & A/\mathfrak{m} \arrow[d]
\\
B \arrow[r]           & B\otimes_A A_\mathfrak{m}^\mathrm{h} \arrow[r] \arrow[r] & (B\otimes_A A/\mathfrak{m})^{\red}                                               
\end{tikzcd}
\]
By~Lemma~\ref{lem:henselian-finite-section},  the bottom right arrow has a section.
Because $B \otimes_A A^\mathrm{h}_\mm$ is a filtered colimit $\varinjlim_{A'} B \otimes_A A'$ as $A'$ ranges over \'etale neighborhoods of $\mm$, and because $(B \otimes_A A/\mm)^\red$ is a finitely presented $K$-algebra, the section
\[ (B \otimes_A A/\mm)^\red \to B \otimes_A A^\mathrm{h}_\mm\]
factors through $B \otimes_A A'$ for some some \'etale neighborhood $A'$ of $\mm$.
Then we have a diagram
\[
\begin{tikzcd}
A \arrow[d] \arrow[r] & A' \arrow[d] \arrow[r] & A/\mathfrak{m} \arrow[d]                         \\
B \arrow[r]           & B\otimes_A A' \arrow[r] \arrow[r] & (B\otimes_A A/\mathfrak{m})^{\red}           
\end{tikzcd}
\]
such that the bottom right arrow continues to admit a section.
This configuration is dual to the desired configuration of schemes.
\end{proof}

The following fact allows us to focus on affine varieties.

\begin{fact}\label{fact:refine}
Suppose that $\cT,\cT^*$ are systems of topologies over $K$ such that the  $\cT$-topology on $K^n=\Aa^n(K)$ refines the $\cT^*$-topology for every $n\ge 1$.
Then $\cT$ refines $\cT^*$.
\end{fact}

Fact~\ref{fact:refine} is \cite[Lemma~4.2]{firstpaper}.

\medskip
We now prove the main result of this section, which is Theorem~C.1.

\begin{theorem}\label{thm:etale-refine}
The $\cE_K$-topology refines the $\cF_K$-topology when $K$ is perfect.
\end{theorem}

\begin{proof}
Suppose that $K$ is perfect.
By Fact~\ref{fact:refine} it suffices to suppose that $V$ is affine and show that any F-subset of $V(K)$ is $\cE_K$-closed.
This is equivalent to the following claim.

\begin{Claim*}
Suppose that $W \to V$ is a finite morphism of affine $K$-varieties, and $p \in V(K)$ is a $K$-point which does not lift to a $K$-point in $W$, i.e., $p$ is not in the image of $W(K) \to V(K)$.
Then there is an \'etale morphism $V' \to V$ such that $p$ is in the image of $f\colon V'(K) \to V(K)$, and $f(V'(K))$ is disjoint from the image of $W(K) \to V(K)$.
\end{Claim*}

Let $F$ and $F^{\red}$ be as in Lemma~\ref{main-lemma}.
Since $K$ is a reduced ring, $F(K) = F^{\red}(K)$.
The fact that $p$ is not in the image of $W(K) \to V(K)$ means that $F(K) = F^{\red}(K) = \varnothing$.
Let $W', V'$ be as in Lemma~\ref{main-lemma}.
Then $V' \to V$ is \'etale, and $p$ lifts to a point in $V'(K)$.
Thus $p$ is in the image of $V'(K) \to V(K)$.
Also, there is a morphism $W' \to F^{\red}$ over $K$.
Therefore $W'(K)$ is empty as $F^{\red}(K)$ is empty.
Then the pullback diagram
\[ \xymatrix{ \varnothing=W'(K) \ar[d] \ar[r] & W(K) \ar[d] \\ V'(K) \ar[r] & V(K)}\]
shows that the image of $W(K) \to V(K)$ is disjoint from the image
of $V'(K) \to V(K)$.
\end{proof}

We give a quick application of Theorem~\ref{thm:etale-refine}. It is of independent interest, will be used in Section~\ref{sec: localhomeo}, and also illustrates how Theorem~\ref{thm:etale-refine} about perfect fields can sometimes be used in combination with Fact~\ref{fact:kins} below to deduce consequences for arbitrary fields.

\begin{fact}\label{fact:kins}
Let $L$ be a purely inseparable extension of $K$.
\begin{enumerate}[leftmargin=*]
\item If $f\colon V \to W$ is an \'etale morphism then $f_L(V(L)) \cap W(K) = f(V(K))$.
\item The inclusion $V(K) \to V(L)$ gives a homeomorphic embedding of the $\cE_K$-topology on $V(K)$ into the $\cE_L$-topology on $V(L)$.
\end{enumerate}
\end{fact}
Here, (2) is \cite[Prop.~5.10]{firstpaper}, and (1) is given in its proof.

\medskip
The Hausdorffness of the \'etale-open topology (Fact~\ref{fact:old-EO}(2)) was first proved in~\cite{firstpaper} and then given another proof in \cite{with-anand}.
It is a key ingredient of the proof that large stable fields are separably closed in~\cite{firstpaper}.
For the proof, we only need to construct two disjoint non-empty $\Sa E_K$-open subsets of $K$, as any invertible affine transformation $K \to K$ is a homeomorphism.
In the first two proofs we constructed disjoint nonempty open subsets $U_1,U_2$ of $K^n$ for some $n\ge 2$, let $\ell$ be a line passing through each $U_i$, and identified $\ell \cap U_1$ and $\ell\cap U_2$ with subsets of $K$ via an affine transformation. Proposition~\ref{prop:hd} gives a third proof with a more canonical pair of open subsets of $K$.

\begin{proposition} \label{prop:hd}
If $K$ is not separably closed, then the $\cE_K$-topology on $K$ is Hausdorff. Moreover, if $K$ is perfect and not algebraically closed, then there is $\cE_K$-open $U$ and $\cF_K$-closed $F$ such that $\varnothing \ne U \subseteq F \subsetneq K$.
\end{proposition}

\begin{proof}
We first prove the second statement.
Suppose, $K$ is perfect and not algebraically closed.
Let $f \in K[x]$ be irreducible with $\deg f>1$, and hence separable.
Let $U$ be the open subvariety of $\Aa^1$ given by $f' \ne 0$.
Note that $U$ is nonempty as $f$ is separable and $\deg f>1$.
Furthermore $f \colon U \to \Aa^1$ is \'etale, hence $f(U(K))$ is a nonempty E-subset of $K$.
On the other hand,  $f(K)$ is an F-set by Fact~\ref{fact:basic-f}(\ref{poly:fini}).
Finally $f(K) \ne K$ as $0 \notin f(K)$.

\medskip
We now prove the first statement.
Suppose $K$ is not separably closed.
If $K$ is perfect, then we have the desired conclusion from the second statement and Theorem~\ref{thm:etale-refine}.
In the general  case, let $L$ be the maximal purely inseparable extension of $K$.
Then $L$ is perfect and not algebraically closed.
Fix separable irreducible $f \in K[x]$ of degree $\ge 2$.
Set $U = \{a \in K : f'(a) \ne 0\}$.
Then $f(U)$ is a non-empty $\Sa E_K$-open subset of $K$.
On the other hand, $f(L)$ is an F-set in $L$ by Fact~\ref{fact:basic-f}(\ref{poly:fini}), hence $\Sa E_L$-closed by Theorem~\ref{thm:etale-refine}.
The $\cE_K$-topology on $K$ is a subspace of the $\cE_L$-topology on $L$ by Fact~\ref{fact:kins}(2), so $K \setminus f(L)$ is $\Sa E_K$-open.  Since $f$ has no zeros in $L$, the set $K \setminus f(L)$ contains 0.
Then $f(U)$ and $K \setminus f(L)$ are the desired disjoint nonempty $\Sa E_K$-open subsets of $K$.
\end{proof}

\section{Boundaries in the \'etale-open topology}\label{section:?}
We apply Theorem~\ref{thm:etale-refine} to study boundaries in the \'etale-open topology. The results in this section will not be used in the later sections.

\begin{lemma}\label{lem:finite-etale}
Suppose that $f\colon V \to W$ is a morphism which is both finite and \'etale.
Then $f(V(K))$ is clopen in the $\cE_K$-topology on $W(K)$.
\end{lemma}

\begin{proof}
Let $L$ be the maximal purely inseparable extension of $K$.
Then $L$ is perfect so by Theorem~\ref{thm:etale-refine} $f_L(V(L))$ is $\cE_L$-clopen.
Apply both parts of Fact~\ref{fact:kins}.
\end{proof}

We remark that, for perfect fields, Lemma~\ref{lem:finite-etale} is an immediate consequence of Theorem~\ref{thm:etale-refine}.
We now turn to a result about boundaries in the \'etale-open topology.

\begin{theorem} \label{thm:small-boundry}
Suppose that $U$ is an E-subset of $V(K)$.
Then the boundary of $U$ in the $\cE_K$-topology is not Zariski dense in $V$. 
\end{theorem}

\begin{proof}
We may suppose that $U \ne \varnothing$.
Let $W \to V$ be an \'etale morphism  with $K$-rational image $U$.
As \'etale morphisms are quasi-finite we apply Fact~\ref{fact:zmt} to factor $W \to
V$ as $W \to Z \to V$, where $W \to Z$ is an open immersion with Zariski dense image and $Z \to V$ is finite.
It follows that $\dim(Z \setminus W) < \dim(W) \le \dim(V)$.
Let $B$ be the image of $Z \setminus W$ in $V$.
Then $B$ is a closed subvariety as finite maps are proper and hence closed.
Furthermore $\dim(B) \le \dim(Z \setminus W) < \dim(V)$.
Hence $B$ is not Zariski dense in $V$.

\medskip
We now show that  the $\cE_K$-boundary of $U$ is contained in $B(K)$.
Let $C$ be the complementary open subvariety $V \setminus B$.
Consider the pullback diagram.
\begin{equation*}
\xymatrix{W \times_V C \ar[r] \ar[d] & Z \times_V C \ar[d] \ar[r] & C \ar[d] \\
W \ar[r] & Z \ar[r] & V}
\end{equation*}
Here $W\times_V C \to C$ is \'etale and $Z\times_V C \to C$ is finite as finite and \'etale morphisms are closed under pullbacks.
The open immersion $W \times_V C \to Z \times_V C$ is an isomorphism, because the complementary closed immersion is $(Z\setminus W) \times_V C \to Z \times_V C$, but $(Z \setminus W) \times_V C$ is empty.
Hence $W \times_V C \to C$ is both finite and \'etale.
Its $K$-rational image is $U \cap C(K)$, which must therefore be an $\cE_K$-clopen subset of $C(K)$ by Lemma~\ref{lem:finite-etale}.
Therefore the $\cE_K$-boundary of $U$ is disjoint from $C(K)$, hence contained in $B(K)$ as desired.
\end{proof}

Recall that a topological space $T$ is \textbf{totally separated} if any two distinct points are separated by a clopen set, and $T$ is \textbf{zero-dimensional} if it admits a basis of clopen sets.
A zero-dimensional Hausdorff topological space is clearly totally separated but there are totally separated Hausdorff topological spaces which are not zero-dimensional~\cite[\S{}127]{counterexamples_in_topology}.

\begin{remark}\label{rem:easy-top}
Let $T$ be a totally separated topological space, with a basis of open sets $\mathcal{B}$ such that every element of $\mathcal{B}$ has finite boundary.
Then $T$ is zero-dimensional.
Indeed, given a basic open set $U \in \mathcal{B}$ and a point $p \in U$, we can find a clopen set $C \ni p$ disjoint from the boundary of $U$, and then $U \cap C$ is a smaller clopen neighborhood of $p$.
\end{remark}

\begin{corollary}\label{cor:zero-dim}
If $K$ is neither separably closed nor isomorphic to $\Rr$ then the $\cE_K$-topology on $K$ is zero-dimensional.
\end{corollary}

Corollary~\ref{cor:zero-dim} follows from Theorem~\ref{thm:small-boundry}, Remark~\ref{rem:easy-top}, Fact~\ref{fact:old-EO}(\ref{old:tsepr}), and the fact that a subset of $K$ which is not Zariski dense in $K$ is finite.

\medskip
Using some standard facts from general topology, we get the following corollary:

\begin{corollary}\label{cor:QQQ}
Suppose that $K$ is countable.
Then the $\cE_K$-topology on $K$ is either discrete, the cofinite topology, or homeomorphic to the order topology on $\Qq$.
\end{corollary}

\begin{proof}
Equip $K$ with the $\cE_K$-topology.
By parts (\ref{old:lrge}) and (\ref{old:sep}) of Fact~\ref{fact:old-EO} it suffices to suppose that $K$ is large and not separably closed, and show that $K$ is homeomorphic to the order topology on $\Qq$.
There are only countably many E-sets, so $K$ is second countable.
By Corollary~\ref{cor:zero-dim} $K$ is zero-dimensional, hence regular.
By the Uryshon metrization theorem~\cite[Thm.~34.1]{Munkres} $K$ is metrizable.
By a theorem of Sierpi\'nski a countable metrizable topological space with no isolated points is homeomorphic to the order topology on $\Qq$~\cite{Sierpinski}.
\end{proof}

\section{The polynomial inverse function theorem for large fields} \label{sec: localhomeo}

We prove Theorem~B from the introduction. 
Let $\poly_m$ be the $K$-variety parameterizing degree $m$ monic polynomials of one variable over $K$.
This is just a copy of $\Aa^m$.
Fact~\ref{fact:poly-multi} is probably known, but we include a proof for lack of a reference.

\begin{fact}
\label{fact:poly-multi}
Fix $n,m \ge 1$ and let $f\colon \poly_n \times \poly_m \to \poly_{n + m}$ be the morphism given by polynomial multiplication.
Then $f$ is finite.
Moreover, $f$ is \'etale at a $K$-point $(g,h)$ if and only if $g,h$ have no common roots in $\kalg$ \textup{(}equivalently do not have a common factor in $K[t]$\textup{)}.
\end{fact}

\begin{proof}
For the first claim  it suffices to show that $f$ affine and proper by Fact~\ref{fact:basic-f}(\ref{afp}).
Now $f$ is affine since its domain and codomain are affine schemes.
By the valuative criterion of properness $f$ is proper if whenever $A$ is a valuation ring over $K$ with fraction field $F$ and $g \in A[t]$ of degree $n + m$ factors as $g = hh^*$ for monic $h,h^* \in F[t]$ of degree $m,n$, respectively, then $h,h^*\in A[t]$.
This is a special case of Gauss's lemma on polynomial factorization.

\medskip
The second claim follows from Fact~\ref{fact:syl} below and the fact that the resultant of two polynomials in $K[t]$ vanishes if and only if the polynomials have a common root in $\kalg$.
\end{proof}

We let $\jac_h(a)$ and $|\!\jac_h(a)|$ be the Jacobian and Jacobian determinant, respectively, of a morphism $h\colon \Aa^m \to \Aa^n$ at $a \in K^m$.

\begin{fact}\label{fact:syl}
Let $f$ and $(g,h)$ be as in Fact~\ref{fact:poly-multi}.
Then $|\!\jac_f(g,h)|$ agrees with the resultant of $g,h$ up to sign.
\end{fact}

\begin{proof}
Set
\begin{align*}
g(t) &= t^n + a_{n - 1} t^{n - 1} + \cdots + a_2 t^2 + a_1 t + a_0 \\
h(t) &= t^m + b_{m - 1} t^{m - 1} + \cdots + b_2 t^2 + b_1 t + b_0
\end{align*}
for $a_i,b_j \in K$.
We compute the Jacobian of $f$ at $(g,h)$.
Set $a_n = 1 = b_m$.
Then 
\begin{align*}
(gh)(t) = \sum_{k = 0}^{m + n} \left( \sum_{i + j = k} a_i b_j\right) t^k
\end{align*}
So $f$ is identified with the morphism $\Aa^{n} \times \Aa^{m} \to \Aa^{m + n}$ given by
\begin{equation*} f(x_0,\ldots,x_{n - 1},y_0,\ldots,y_{m - 1}) = \left(x_0 y_0, \ldots, \sum_{i + j = k} x_i y_j, \ldots, y_{m - 1} + x_{n - 1} \right) \end{equation*}
Therefore $\der f_k/\der x_\ell (g,h)$ equals $b_{k - \ell}$ when $k - \ell \ge 0$ and $0$ otherwise, and $\der f_k/\der y_\ell (g,h)$ equals $a_{k - \ell}$ when $k - \ell \ge 0$ and $0$ otherwise.
So  $\jac_f(g,h)$ is, up to possibly permuting rows, the Sylvester matrix of $(g,h)$.
Hence $|\!\jac_f(g,h)|$ is, up to sign, the resultant of $(g,h)$.
\end{proof}

Lemma~\ref{lem:ez-fc-1} is an easy extension of the second statement in Proposition~\ref{prop:hd}.

\begin{lemma}
\label{lem:ez-fc-1}
Let $K$ be perfect and not algebraically closed, $A \subseteq K$ be finite, and $b \in K \setminus A$.
There is an $\cE_K$-open $U\subseteq K$ and an $\cF_K$-closed $F\subseteq K$ so that $b \in U \subseteq F$ and $A \cap F = \varnothing$.
\end{lemma}

\begin{proof}
After possibly translating we may suppose that $b = 0$.
Applying Proposition~\ref{prop:hd} and translating, we obtain an $\cE_K$-open $U$ and an $\cF_K$-closed $F$ such that $0 \in U \subseteq F \subsetneq K$.
For each $a \in A$ fix $c_a \in K^\times$ such that $a \notin c_a F$.
Then $F^* = \bigcap_{a \in A} c_a F$ is $\cF_K$-closed and $U^* = \bigcap_{a \in A} aU$ is $\cE_K$-open.
We have $0 \in U^* \subseteq F^*$ and $A \cap F^* = \varnothing$.
\end{proof}

The next lemma can be seen as a weak form of continuity of roots for the \'etale-open topology.

\begin{lemma}
\label{lem:roots}
Suppose that $K$ is perfect and not algebraically closed.
Fix $f \in \poly_d(K)$ and let $a \in K$ be a simple root of $f$.
Then there are $\cE_K$-open neighborhoods $U$ of $a$ and $O$ of $f$ such that any $h \in O$ has a unique root in $U$, and the root is simple.
\end{lemma}

\begin{proof}
First observe that the map $K \to K, b \mapsto b - a$ and the map $\poly_d(K) \to \poly_d(K)$ given by $f(t) \mapsto f(t - a)$ are both $\cE_K$-homeomorphisms.
Hence may suppose that $a=0$.
Then $f(0) = 0\ne f'(0)$.
So $f(t) = t g(t)$ for some $g \in K[t]$ with $g(0) \ne 0$.
Applying Lemma~\ref{lem:ez-fc-1} we fix an $\cE_K$-open set $U$ and an $\cF_K$-closed set $F$ such that $0 \in U \subseteq F$ and $F$ does not contain a  root of $g$.
By Fact~\ref{fact:prod} $F \times \poly_{d - 2}(K)$ is an $\cF_K$-closed subset of $\poly_1(K) \times \poly_{d - 2}(K)$.
Let $H \subseteq \poly_{d - 1}(K)$ be the image of $F \times \poly_{d - 2}(K)$ under the multiplication map.
So $H$ is the set of degree $d - 1$ monic polynomials that have a root in $F$.
Lemma~\ref{lem:system-2} and Theorem~\ref{thm:etale-refine} together show that $H$ is $\cE_K$-closed.
Moreover, $g$ is not in $H$ as $g$ does not have a root in $F$.
Let $S$ be the complement of $H$ in $\poly_{d - 1}(K)$, so $S$ is $\cE_K$-open.
Hence $S \times U$ is an $\cE_K$-open neighborhood of $(g,0)$ in $\poly_{d - 1}(K) \times \poly_1(K)$.
As $g(0) \ne 0$, Fact~\ref{fact:poly-multi} implies that the multiplication morphism $\poly_{d - 1} \times \poly_1 \to \poly_d$ is \'etale at $(g,0)$.
Let $E \subseteq \poly_{d-1} \times \poly_1$ be the \'etale locus of the multiplication morphism.
Then $(S \times U) \cap E(K)$ is $\cE_K$-open in $E(K)$.
Let $O$ be the image of $(S \times U) \cap E(K)$ under the restriction of the multiplication morphism to $E(K) \to \poly_d(K)$.
Then $O$ is $\cE_K$-open by Definition~\ref{def:E}.
Moreover, $O$ contains $f$, the image of $(g,0)$.
Thus $O$ is an $\cE_K$-open neighborhood of $f$ which is contained in the image of $S \times U$ under the multiplication map $\poly_{d-1}(K) \times \poly_1(K) \to \poly_d(K)$.
Suppose $h \in O$.
Then $h(t) = (t - a) h^*(t)$ where $a \in U$ and $h^*(t)$ has no roots in $F$, hence $h$ has exactly one root in $U$, and this root is simple.
\end{proof}

We now prove the main result of this section, which is the first part of Theorem~B in the introduction.

\begin{theorem}
\label{thm:local-homeo-2}
Suppose that $K$ is not separably closed and $f \colon V \to W$ is \'etale.
Then $V(K) \to W(K)$ is a local homeomorphism in the $\cE_K$-topology.
\end{theorem}

We let $\Sa O(W)$ be the ring of regular functions on $W$.

\begin{proof}
As $V(K) \to W(K)$ is continuous and open, it suffices to show that it is locally injective.
We first reduce to the case when $K$ is perfect.
Let $L/K$ be the maximal purely inseparable extension of $K$, so $L$ is perfect.
Let $V_L \to W_L$ be the base change of $V \to W$.
Then $V_L \to W_L$ is \'etale as \'etale morphisms are closed under base change.
The $\cE_K$-topology on $V(K)$ agrees with the topology induced by the $\cE_L$-topology on $V_L(L)$ by Fact~\ref{fact:kins}.
Hence it is enough to show that $V_L(L) \to W_L(L)$ is locally injective in the $\cE_L$-topology.
Thus we may suppose that $K$ is perfect.

\medskip
By Fact~\ref{fac: sdetalevsetale} we may suppose that $V\to W$ is standard \'etale.
Hence we suppose that $V$ is a subvariety of $W \times \Aa^1$ given by $f = 0 \ne g$ for some $f,g \in \Sa O(W)[y]$ such that $f$ is monic in $y$ and $\der f/\der y$ does not vanish on $V$, and suppose that $V \to W$ is the projection.

\medskip
Fix $(a,b) \in V(K)$.
We show that $f$ is locally injective at $(a,b)$.
Let $f_c \in K[y]$ be given by $f_c(y) = f(c,y)$ for each $c \in W(K)$.
Every $f_c$ is monic of degree $d$ for some $d$ as $f$ is monic.
Let $\gamma \colon W(K) \to \poly_d(K)$ be given by $\gamma(c) = f_c$.
Then $\gamma$ is a continuous map.
We have $f'_a(b) = \der f/\der y (a,b) \ne 0$, where the derivative $f'_b$ is taken in $K[y]$.
Applying Lemma~\ref{lem:roots} we obtain $\cE_K$-open neighborhoods $b \in U \subseteq K$ and $f_a \in O^* \subseteq \poly_d(K)$ so that any $h \in O^*$ has a unique root in $U$.
Let $O = \gamma^{-1}(O^*)$.
Then $O$ is an $\cE_K$-open neighborhood $a \in O \subseteq W(K)$, and $f_c$ has a unique root in $U$ when $c \in O$.
Then $V(K) \cap (O \times U)$ is an $\cE_K$-open neighborhood of $(a,b) \in V(K)$ and the restriction of $f$ to $V(K) \cap (O \times U) \to O$ is injective.
\end{proof}

Fact~\ref{fact:local-coordinate} below is \cite[Lemma~054L]{stacks-project}.  We need to assume irreducibility simply to ensure that the local dimension of $V$ at $p \in V$ does not depend on $p$.

\begin{fact}\label{fact:local-coordinate}
If $V$ is smooth and irreducible then any $p\in V$ is contained in an open subvariety $W$ of $V$ which admits an \'etale morphism $W\to \Aa^d$ for $d = \dim V$
\end{fact}

Corollary~\ref{cor:manifold} below follows from Theorem~\ref{thm:local-homeo-2} and Fact~\ref{fact:local-coordinate}. It is the second part of Theorem~B from the introduction. 

\begin{corollary}\label{cor:manifold}
If $V$ is smooth and irreducible then the $\cE_K$-topology on $V(K)$ is locally homeomorphic to the $\cE_K$-topology on $K^d$ for $d$ the dimension of $V$.
\end{corollary}

\section{Nash functions over large fields}\label{section:nashh-basic}

The \'etale-open topology introduced by us is different from Grothendieck's much more famous \'etale topology.
In particular, the former is a topology in the classical sense while the latter is not.
We nevertheless establish a connection between them for large fields (Theorem~\ref{thm:milne}).
This uses Theorem~B from the last section and will be used to prove Theorem~A 
in the next section.

\medskip
Let $U$ be an open semialgebraic subset of $\Rr^n$ and $f$ be a function $U \to \Rr$.
Then $f$ is \textit{Nash} if it is $C^\infty$ and semialgebraic, equivalently if it is real analytic and algebraic over $\Rr[x_1,\ldots,x_n]$~\cite[\S{}8]{real-algebraic-geometry}.
Artin and Mazur showed that $f$ is Nash if and only if it can be expressed as a composition of an inverse of an \'etale map with an algebraic map~\cite[Thm.~8.4.4]{real-algebraic-geometry}.
This allows us to define Nash functions over non-separably closed large fields.
\textbf{Throughout this section $K$ is large and not separably closed.}

\medskip
Suppose that $V_i$ is a smooth $K$-variety and $O_i$ is a nonempty $\cE_K$-open subset of $V_i(K)$ for $i = 1,2$.
A function $f \colon O_1 \to O_2$ is a \textbf{Nash map} if there is a $K$-variety $X$, \'etale morphism $e \colon X \to V_1$, morphism $h \colon X \to V_2$, and $\cE_K$-open subset $P$ of $X(K)$ such that $e$ restricts to a homeomorphism $P \to O_1$, $h$ maps $P$ into $O_2$, and $f = h \circ (e|_P)^{-1}$.
A \textbf{Nash function} is a Nash map with codomain $K$.
Note that a Nash map is $\cE_K$-continuous.
We say that $(X,e,h,P)$ is \textbf{Nash data} for $f$.
Note that $X$ is necessarily smooth as $V$ is smooth and $e$ is \'etale.
Furthermore note that $V(K) \to W(K)$ is a Nash map for any morphism $V \to W$ between smooth varieties as the identity $V \to V$ is \'etale.
The following proposition shows that Nash maps form a category.

\begin{proposition}
\label{prop:composition}
Nash maps are closed under composition.
\end{proposition}

\begin{proof}
Let $V_i$ be a smooth $K$-variety and $O_i$ be a nonempty $\cE_K$-open subset of $V_i(K)$ for $i=1,2,3$.
Let $f_1 \colon O_1 \to O_2$ and $f_2 \colon O_2 \to O_3$ be Nash maps.
Given $i = 1,2$ we let $(X_i,e_i,h_i,P_i)$ be Nash data for $f_i$.
Let $\uppi_i \colon X_1 \times_{V_2} X_2 \to X_i$ be the projection for $i\in \{1,2\}$.
Note that $\uppi_1$ is \'etale as it is the base change of an \'etale morphism.
Hence $e_1 \circ \uppi_1$ is an \'etale morphism $X_1 \times_{V_2} X_2 \to V_1$.
We have the following diagram where the vertical arrows are \'etale morphisms and the horizontal arrows are morphisms.
\[ \xymatrix{
 X_1 \times_{V_2} X_2 \ar[r]^{\quad\uppi_2} \ar[d]^{\uppi_1} & X_2 \ar[d]^{e_2} \ar[r]^{h_2} & V_3 \\  X_1 \ar[d]^{e_1}
\ar[r]^{h_1} & V_2  \\ V_1 } \]
Let $P_1 \times_{O_2} P_2$ be the set-theoretic fiber product of $P_1$ and $P_2$ over $O_2$.
By definition of the fiber product we have $P_1 \times_{O_2} P_2 = \uppi^{-1}_1(P_1) \cap \uppi^{-1}_2(P_2)$ and hence $P_1 \times_{O_2} P_2$ is an $\cE_K$-open subset of $(X_1 \times_{V_2} X_2)(K)$.
We show that the projection $P_1 \times_{O_2} P_2 \to P_1$ is injective and hence an $\cE_K$-homeomorphism.
Fix elements $(a,b),(a,b^*) \in P_1 \times_{O_2} P_2$, so $a\in P_1$ and $b,b^*\in P_2$.
Then $e_2(b) = h_1(a) = e_2(b^*)$.
As $a \in P_1$ we have $h_1(a) \in O_2$, so $b = b^*$ as $e_2$ gives a bijection $P_2 \to O_2$.
Let $d_1,d_2, \rho$ be the restriction of $e_1,e_2,\uppi_1$ to $O_1,O_2, P_1 \times_{O_2} P_2$, respectively.
Therefore we obtain the following diagram of Nash maps.
\[ \xymatrix{
P_1 \times_{O_2} P_2 \ar[r]^{\quad \uppi_2}
& P_2 \ar[r]^{h_2} & O_3 \\  P_1 \ar[u]_{\rho^{-1}}
\ar[r]^{h_1} & O_2 \ar[ur]_{f_2} \ar[u]_{d^{-1}_2}  \\ O_1 \ar[ur]_{f_1} \ar[u]_{d^{-1}_1} } \]
It is easy to see that this diagram commutes.
It follows that
$$ f_2 \circ f_1 = h_2 \circ d^{-1}_2 \circ h_1 \circ d^{-1}_1 = h_2 \circ \uppi_2 \circ \rho^{-1} \circ d^{-1}_1 = (h_2 \circ \uppi_2) \circ (d_1 \circ \rho)^{-1}.$$
Hence $f_2 \circ f_1$ is Nash.
\end{proof}

Let $O$ be a nonempty \'etale-open subset of $K^m$ and $f \colon O \to K^n$ be a Nash map.
Then each composition of $f$ with a coordinate projection $K^n \to K$ is a composition of Nash maps and is hence Nash.
Proposition~\ref{prop:nash-product} gives the converse.

\begin{proposition}
\label{prop:nash-product}
Suppose that $V$, $W_1,\ldots,W_n$ are smooth $K$-varieties, $O$ is a nonempty $\cE_K$-open subset of $V(K)$, and $f_i \colon O \to W_i(K)$ is a Nash map for each $i \in \{1,\ldots,n\}$.
Let $f \colon O \to (W_1 \times \ldots \times W_n)(K)$ be given by $f = (f_1,\ldots,f_n)$.
Then $f$ is a Nash map.
\end{proposition}

Equivalently the set-theoretic product is the product in the category of Nash maps.

\begin{proof}
An obvious induction reduces to the case when $n = 2$.
For each $i \in \{1,2\}$ let $(X_i,e_i,h_i,P_i)$ be Nash data for $f_i$.
Let $\uppi_i \colon X_1 \times_V X_2 \to X_i$ be the projection for $i \in \{1,2\}$ and $e \colon X_1 \times_V X_2 \to V$ be the canonical morphism.
Recall that $X_1 \times_V X_2$ is the closed subvariety of $X_1 \times X_2$ given by $e_1(x) = e_2(y)$, and $e$ is given by $e(x,y) = e_1(x)$.
Now $e$ is \'etale.
Let $h \colon X_1 \times_V X_2 \to W_1 \times W_2$ be the morphism given by $h(x,y) = (h_1(x),h_2(y))$.
\begin{center}
\begin{tikzcd}
  & X_1 \arrow[ld, "e_1"'] \arrow[r, "h_1"] & W_1                                                                                     &                \\
V &                                         & X_1 \times_V X_2 \arrow[lu, "\uppi_1"'] \arrow[ld, "\uppi_2"'] \arrow[ll, "e"'] \arrow[r, "h"] & W_1 \times W_2 \\
  & X_2 \arrow[lu, "e_2"] \arrow[r, "h_2"]  & W_2                                                                                     &               
\end{tikzcd}
\end{center}
As in the proof of Proposition~\ref{prop:composition} $P_1 \times_{O} P_2$ is an $\cE_K$-open subset of $(X_1 \times_V X_2)(K)$.
Each $e_i$ gives a bijection $P_i \to O$.
Hence for every $a_1 \in P_1$ there is a unique $a_2 \in P_2$ such that $e_1(a_1) = e_2(a_2)$.
Equivalently for every $a_1 \in P_1$ we have $(a_1,a_2)\in P_1 \times_O P_2$ for a unique $a_2 \in P_2$.
Hence $\uppi_1$ gives a bijection $P_1 \times_O P_2 \to P_1$ and so $e$ gives a bijection $P_1 \times_O P_2 \to O$.
As $e$ is \'etale this bijection is a homeomorphism.
Hence $(X_1 \times_V X_2, e , h , P_1 \times_O P_2)$ is Nash data for a Nash map $f \colon O \to (W_1 \times W_2)(K)$.
It remains to show that $f = (f_1,f_2)$.
We leave this to the reader.
\end{proof}

\begin{proposition}
\label{prop:nash-ring-0}
Suppose $V$ is a smooth irreducible $K$-variety and $O \subseteq V(K)$ is nonempty $\cE_K$-open.
Let $\Sa N(O)$ be the set of Nash functions on $O$.
Then $\Sa N(O)$ is closed under pointwise addition and multiplication.
If $f \in \Sa N(O)$ is everywhere nonzero on $O$ then $1/f \in \Sa N(O)$.
\end{proposition}

\begin{proof}
Fix $f,g \in \Sa N(O)$.
By Proposition~\ref{prop:nash-product} $(f,g)$ is a Nash map $O \to K^2$.
Now $f+g$, $fg$ is the composition of $(f,g)$ with the map $K^2\to K$ given by $(x,y)\mapsto x + y$, $(x,y)\mapsto xy$, respectively.
Hence $f+g$ and $fg$ are compositions of Nash maps and hence Nash.
A similar argument handles $1/f$.
\end{proof}

Note that a constant map $O \to K$ is trivially Nash.
We therefore consider $\Sa N(O)$ to be a $K$-algebra in the natural way.

\begin{lemma}\label{lem:O->N}
Suppose that $V$ is a smooth irreducible $K$-variety and $O$ is a nonempty $\cE_K$-open subset of $V(K)$.
Let $\Sa O(V)$ be the ring of regular functions on $V$.
Then $f \mapsto f|_{O}$ gives a $K$-algebra embedding $\Sa O(V) \to \Sa N(O)$.
\end{lemma}

\begin{proof}
Clearly $f \mapsto h|_{O}$ gives a $K$-algebra morphism.
Let $f, g \in \Sa O(V)$ be distinct.
Then the open subvariety of $V$ given by $f \ne g$ intersects $O$ by Fact~\ref{fact:z-dense}, hence $f|_{O} \ne g|_{O}$.
\end{proof}

If $O^*$ is an $\cE_K$-open subset of $O$ and $f\colon O \to K$ is a Nash function then the restriction of $f$ to $O^*$ is also Nash.
It is easy to see that restriction gives a $K$-algebra morphism $\Sa N(O) \to \Sa N(O^*)$.

\medskip
Given $p \in V(K)$ we let $\Sa N_p$ be the $K$-algebra of germs of Nash functions at $p$, i.e., the inverse limit of $\Sa N(O)$ where $O$ ranges over $\cE_K$-open neighborhoods of $p$.

\begin{proposition}
\label{prop:nash-ring}
Suppose that $V$ is a smooth $K$-variety and $p \in V(K)$.
Then $f \in \Sa N_p$ is invertible in $\Sa N_p$ if and only if $f(p) \ne 0$.
Hence $\Sa N_p$ is a local ring with maximal ideal $\{f \in \Sa N_p : f(p) = 0\}$, residue field $K$, and residue map $f\mapsto f(p)$.
\end{proposition}

\begin{proof}
The second claim is immediate from the first.
The left to right implication of the first claim is clear.
Suppose $f(p) \ne 0$.
As $f$ is $\cE_K$-continuous there is an $\cE_K$-open neighborhood $O$ of $p$ so that $f$ is defined and does not vanish on $O$.
Apply Proposition~\ref{prop:nash-ring-0}.
\end{proof}

Let $\Sa O_p$ be the local ring of a variety $V$ as a point $p\in V$.

\begin{lemma}
\label{lem:nash-ring}
Suppose that $V$ is a smooth $K$-variety and $p \in V(K)$.
Then $h \mapsto h|_{V(K)}$ gives a $K$-algebra embedding $\Sa O_p \to \Sa N_p$.
\end{lemma}

Lemma~\ref{lem:nash-ring} follows directly from Lemma~\ref{lem:O->N}.

\begin{lemma}
\label{lem:nash-ring-1}
Suppose $e : X \to V$ is an \'etale morphism of smooth $K$-varieties, $q \in X(K)$, and $p = e(q)$.
Then $h \mapsto h \circ e$ gives a $K$-algebra isomorphism $\Sa N_p \to \Sa N_q$.
\end{lemma}

Lemma~\ref{lem:nash-ring-1} follows easily from Theorem~\ref{thm:local-homeo-2} and Proposition~\ref{prop:composition}, and is left as an exercise to the reader.

\medskip
We first review some definitions.
An \textbf{\'etale neighborhood} $e : (X,q) \to (V,p)$ of $p\in V$ is a $K$-variety $X$, $q \in X(K)$, and an \'etale morphism $e \colon X \to V$ such that $e(q) = p$.
We say that $e \colon (X,q) \to (V,p)$ refines $e^* \colon (X^*,q^*) \to (V,p)$ if $e$ factors as $X \to X^* \to V$ where $X^* \to V$ is $e^*$.
Taking fiber products shows that \'etale neighborhoods form a directed class.
The \textbf{\'etale local ring} $\Sa O^{\acute{e}t}_{V,p}$ of $p$ on $V$ is the filtered colimit of the coordinate ring of $X$'s. Fact~\ref{fact:etalelocalring} is~~\cite[Lemma~05KS]{stacks-project}.

\begin{fact} \label{fact:etalelocalring}
The \'etale local ring is isomorphic to the henselization $\Sa O^\mathrm{h}_p$ of $\Sa O_p$
\end{fact}

We now prove the main result of this section.

\begin{theorem} \label{thm:milne}
Suppose that $V$ is a smooth $K$-variety and $p \in V(K)$.
Then $\Sa N_p$ is canonically isomorphic as a $K$-algebra to the \'etale local ring $\Sa O^{\acute{e}t}_{V,p}$ of $p$ on $V$. Hence, $\Sa N_p$ is also isomorphic to the henselization $\Sa O^\mathrm{h}_p$ of $\Sa O_p$ and to $\palg$ where $m$ is the dimension of the irreducible component of $V$ containing $m$.
\end{theorem}

\begin{proof}
Fix an \'etale neighborhood $g \colon (X,q) \to (V,p)$ of $p$.
Let $\uptau_g$ be the $K$-algebra embedding $\Sa O_q \to \Sa N_p$ given by composing the embedding $\Sa O_q \to \Sa N_q$ produced in Lemma~\ref{lem:nash-ring} with the $K$-algebra isomorphism $\Sa N_q \to \Sa N_p$ produced by Lemma~\ref{lem:nash-ring-1}.
Therefore $(\uptau_e)_{e \colon (X,q) \to (V,p)}$ gives a $K$-algebra embedding $\uptau \colon \Sa O^\mathrm{h}_p \to \Sa N_p$.
It follows by definition of a Nash function that every element of $\Sa N_p$ is the image of some $\uptau_e$.
Hence $\uptau$ is a $K$-algebra isomorphism.
This proves the first claim of Theorem~\ref{thm:milne}.

\medskip
The first part of the second claim follows immediately from Fact~\ref{fact:etalelocalring}. It remains to show the second part of the second claim. When $V$ is affine space, the first claim then shows that $\Sa N_p$ is the henselization of the localization $K[t_1,\ldots,t_m]_{(t_1,\ldots,t_m)}$, which is $\palg$ by Fact~\ref{fact:palg}.
The general case of the second claim follows by Fact~\ref{fact:local-coordinate} and Lemma~\ref{lem:nash-ring-1}.
\end{proof}

\begin{remark}
For our purposes we only need to know that $\Sa N_p$ is a henselian local domain.
We give a sketch of a direct proof.
Fix $f \in \Sa N_p[x]$ and suppose that $\alpha \in K$ is a simple root of $\bar{f}$ where $\bar{f}$ is $f$ modulo the maximal ideal.
We may suppose that $f$ is monic and let
$$ f(x) = x^{d + 1} + g_d x^d + \cdots + g_1 x + g_0 \quad \text{for  } g_0,\ldots,g_d \in \Sa N_p.$$
So $\bar{f}(x) = x^{d + 1} + g_d(p) x^d + \cdots + g_1(p) x + g_0(p)$.
Fix an $\Sa E_K$-open neighborhood $O$ of $p$ on which every $g_i$ is defined.
Let $(X_i, e_i, h_i, P_i)$ be Nash data for $g_i$ for each $i = 0, \ldots, d$.
Let $X$ be the fiber product of the $X_i$ over $V$ and $P$ be the set-theoretic fiber product of the $P_i$ over $O$.
Let $e^* \colon X \to V$ and $e^*_i \colon X \to X_i$ be the natural morphisms and let $h^*_i = h_i \circ e^*_i$ for each $i$.
Then each $(X, e^*, h^*_i, P)$ is also Nash data for $g_i$.
Take $q = (e^*|_{P})^{-1}(p)$.
The isomorphism $\Sa N_q \to \Sa N_p$  given by Lemma~\ref{lem:nash-ring-1} takes each $h^*_i$ to $g_i$.
Hence after replacing $\Sa N_p$ with $\Sa N_q$ we may suppose that every $g_i$ is in the image of the embedding $\Sa O_p \to \Sa N_p$ given by Lemma~\ref{lem:nash-ring}.
Identify each $g_i$ with the corresponding element of $\Sa O_p$.
Let $U$ be an open subvariety of $V$ on which each $g_i$ is defined.
Let $W$ be the subvariety of $U \times \Aa^1$ given by
$$ x^{d + 1} + g_d x^d + \cdots + g_1 x + h_0 = 0 \ne \frac{\der}{\der x} [x^{d + 1} + g_d x^d + \cdots + g_1 x + h_0]. $$
Then $(p,\alpha) \in W(K)$ and the projection $W \to U$ is a standard \'etale morphism.
After possibly shrinking $O$ we may suppose that there is an $\Sa E_K$-open neighborhood $O^*$ of $(p,\alpha)$ in $W(K)$ such that the projection $W(K) \to U(K)$ gives an $\Sa E_K$-homeomorphism $O^* \to O$.
Let $\gamma$ be the composition of the inverse $O \to O^*$ with the projection $U(K) \times K \to K$.
Then $\gamma$ is an element of $\Sa N_p$, $\gamma$ is a root of $f$, and $\gamma(p) = \alpha$.
\end{remark}

\section{Large implies henselian}\label{section:l->h}
In this section, we prove Theorem~A from the introduction, that every large field $K$ is elementarily equivalent to the fraction field of a proper henselian local domain. The case when $K$ is separably closed is Proposition~\ref{prop:separablecaseofA} below. The case when $K$ is \emph{not} separably closed follows from Theorem~\ref{thm:infin3}.

\begin{proposition} \label{prop:separablecaseofA}
Every separably closed field $K$ is elementarily equivalent to the fraction field of a proper henselian local domain.
\end{proposition}

\begin{proof}
After possibly passing to an elementary extension, we may suppose that $K$ is not algebraic over its prime subfield. Hence, there is a non-trivial valuation $v$ on $K$ and $K$ is the fraction field of the valuation ring $R$ of $v$.
Any valuation on a separably closed field is henselian~\cite[Lemma~4.1.1]{EP-value}, hence $R$ is a henselian local domain.
\end{proof}

\medskip
Suppose that $K$ is large and not separably closed and let $\Kk$ be an elementary extension of $K$.
Given a $K$-definable set $X\subseteq K^n$ let $X(\Kk)$ be the subset of $\Kk^n$ defined by any formula defining $X$ and given a $K$-definable function $f\colon X \to Y$ let $f_\Kk\colon X(\Kk)\to Y(\Kk)$ be the $\Kk$-definable function defined by any formula defining $f$.
We say that $a \in \Kk$ is $K$-infinitesimal if and only if $a\in U(\Kk)$ for every $E$-subset $U$ of $K$ containing $0$.
More generally, an indexed family $(a_i)_{i\in I}$ of elements of $\Kk$ is \textbf{$K$-infinitesimal} if $(a_{i_1},\ldots,a_{i_m})\in U(\Kk)$ for every $i_1,\ldots,i_m \in I$ and E-subset $U$ of $K^m$  containing the origin.
Finally, we say that a subset $A$ of $\Kk$ is $K$-infinitesimal if $(a_1,\ldots,a_n)$ is $K$-infinitesimal for any $a_1,\ldots,a_n \in A$, equivalently if any tuple enumerating $A$ is $K$-infinitesimal.

\medskip
Note that $(a_i)_{i\in I}$ might not be $K$-infinitesimal even when every $a_i$ is $K$-infinitesimal, because the $\cE_K$-topology on $K^m$ may not agree with the product topology coming from the $\cE_K$-topology on $K$.
If the $\cE_K$-topology on $K^m$ does not agree with the product topology then by a standard application of model-theoretic compactness we can produce $\Kk$ and $a\in \Kk^m$ such that each coordinate of $a$ is $K$-infinitesimal and $a$ is not.

\medskip
We recall an easy fact about the \'etale-open topology.

\begin{fact}\label{fact:projections}
Let $U$ be an $\cE_K$-open subset of $V(K) \times W(K)$.
Then $\{ b \in W(K) : (a,b) \in U\}$ is $\cE_K$-open for any $a \in V(K)$.
\end{fact}

\begin{proof}
The map $V(K) \to V(K) \times W(K)$ given by $b \mapsto (a,b)$ is $\cE_K$-continuous.
\end{proof}

A union of $K$-infinitesimal sets is not necessarily $K$-infinitesimal, but we do have the following Lemma~\ref{lem:infin1}, which allows us to construct suitable $K$-infinitesimal sets later.

\begin{lemma}\label{lem:infin1}
Suppose that $K$ is large and not separably closed and $K \preceq \Kk$.
Fix a linearly ordered index set $I$ with minimal element $0$ and let $(K_i)_{i\in I}$ be an increasing sequence of elementary subfields of $\Kk$ with $K_0 = K$.
For each $i \in I$ let $A_i$ be a $K_i$-infinitesimal subset of $\Kk$ so that $A_i \subseteq K_j$ for $j > i$.
Then $\bigcup_{i \in I} A_i$ is a $K$-infinitesimal set.
\end{lemma}

\begin{proof}
By the definitions it suffices to fix $i_0,\ldots,i_n \in I$, $m\ge 1$, and elements $a^i_1,\ldots,a^i_{m}$ of $A_i$ for each $i \in \{i_0,\ldots,i_n\}$ and show $(a^i_j : i\in \{i_0,\ldots,i_n\}, j \in \{1,\ldots m\})$ is $K$-infinitesimal.
Hence we suppose that $I = \{0,\ldots,n\}$ for some $n$.
We apply induction on $n$.
The base case $n = 0$ is trivial.
Fix $n\ge 1$ and suppose that $a' = (a^1_1,\ldots,a^1_{m},\ldots,a^{n-1}_1,\ldots, a^{n-1}_m)$ is $K$-infinitesimal.
Declare $a'' = (a^n_1,\ldots,a^n_{m})$.
Fix an E-subset $U$ of $K^{nm}$ containing the origin.
We show that $(a', a'')\in U(\Kk)$.
By Fact~\ref{fact:projections} $\{ c \in K^{(n-1)m} : (c,0,\ldots,0) \in U\}$ is a $K$-definable $\cE_K$-open set containing the origin.
Hence as $a'$ is $K$-infinitesimal we have $(a',0,\ldots,0) \in U(\Kk)$.
Again by Fact~\ref{fact:projections} $O = \{c \in K^{m}_{n} : (a',c) \in U(K_{n})\}$ is a $K_{n}$-definable $\cE_{K_{n}}$-neighborhood of the origin.
Hence $a'' \in O(\Kk)$ as $a_n$ is $K_n$-infinitesimal, so $(a',a'') \in U(\Kk)$.
\end{proof}

Proposition~\ref{prop:eval} below allows us to produce henselian local rings from $K$-infinitesimal sets.

\begin{proposition}\label{prop:eval}
Suppose that $K$ is large and not separably closed, let $\Kk$ be an elementary extension of $K$, and $\varepsilon = (\varepsilon_i)_{i\in I} \in \Kk^I$ be a $K$-infinitesimal tuple for an index set $I$.
There is a canonical $K$-algebra morphism $\eval_\varepsilon\colon \palgI \to \Kk$ so that $\eval_\varepsilon(t_i) = \varepsilon_i$ for every $i\in I$.  Moreover, the image is contained in the relative algebraic closure of $K(\varepsilon_i)_{i \in I}$ in $\Kk$.  Finally, the image of $\eval_\varepsilon \colon \palgI \to \Kk$ is a henselian local ring with residue field $K$.
\end{proposition}

\begin{proof}
Fix $f\in \palgI$.
Then $f\in K[t_{i_1},\ldots,t_{i_m}]^\mathrm{h}$ for some $i_1,\ldots,i_m\in I$.
Applying Theorem~\ref{thm:milne} we identify $K[t_{i_1},\ldots,t_{i_m}]^\mathrm{h}$ with the ring of germs of Nash functions on $K^m$ at the origin and therefore consider $f$ to be such a germ.
Let $U$ be an E-subset of $K^m$ containing the origin on which $f$ is defined.
Then $(\varepsilon_{i_1},\ldots,\varepsilon_{i_m})$ is in $U(\Kk)$ as $\varepsilon$ is $K$-infinitesimal, hence $f_\Kk(\varepsilon_{i_1},\ldots,\varepsilon_{i_m})$ is defined.  By definition of Nash functions, $f_\Kk(\varepsilon_{i_1},\ldots,\varepsilon_{i_m})$ is algebraic over $K(\varepsilon_{i_1},\ldots,\varepsilon_{i_m})$.
Set $\eval_\varepsilon(f) = f_\Kk(\varepsilon_{i_1},\ldots,\varepsilon_{i_m})$.
It is easy to see that $\eval_\varepsilon$ is a $K$-algebra morphism.

\medskip
The last assertion follows from the fact that the image is the quotient of the henselian local ring $\palgI$ by an ideal.
\end{proof}

Suppose that $K$ is large and not separably closed and let $\Kk$ be an elementary extension of $K$.
Given a nonempty $K$-infinitesimal  $E \subseteq \Kk$ of cardinality $\kappa$, we let $K[E]^\mathrm{h}$ be the image of $\eval_\varepsilon \colon \palgkap \to \Kk$ where $\varepsilon = (\varepsilon_i)_{i < \kappa}$ is any tuple enumerating $E$.  Then $K[E]^\mathrm{h}$ is a henselian local ring with residue field $K$.
We record a number of easy facts leaving the proofs to the reader.

\begin{lemma} \label{lem: obviousepsilonimage}
Under the same assumptions as the preceding paragraph, we have the following:
\begin{enumerate}[leftmargin=*]
\item $E$ is a subset of the maximal ideal of $K[E]^\mathrm{h}$.
\item If  $E \neq \{ 0\}$, then $K[E]^\mathrm{h}$ is not a field.
\item $K[E]^\mathrm{h}$  is the filtered union of  $K[E_0]^\mathrm{h}$ as $E_0$ ranges over finite subsets of $E$.
\item $K[E]^\mathrm{h}$ is contained in the relative algebraic closure of $K \cup E$ in $\Kk$.
\end{enumerate}
\end{lemma}

The proof of Lemma~\ref{lem:infin0} is also left to the reader.

\begin{lemma}\label{lem:infin0}
Suppose that $K$ is large and not separably closed, $\Kk$ is an elementary extension of $K$, $f\colon K^m \to K^n$ is a $K$-definable function which is $\cE_K$-continuous and takes the origin to the origin, and $\varepsilon \in \Kk^m$ is  $K$-infinitesimal.
Then $f_\Kk(\varepsilon) \in \Kk^n$ is also $K$-infinitesimal.
\end{lemma}

A special consequence is needed for our purposes.

\begin{corollary} \label{cor:infin.5}
Suppose that $K$ is large and not separably closed, and $\varepsilon$ is a non-zero $K$-infinitesimal element of an elementary extension $\Kk$ of $K$.
Then $\varepsilon K$ is a $K$-infinitesimal subset of $\Kk$.
\end{corollary}

\begin{proof}
It suffices to show  that $(\varepsilon b_1,\ldots,\varepsilon b_n)$ is $K$-infinitesimal for any $b_1,\ldots,b_n \in K$.
Apply Lemma~\ref{lem:infin0} with $f\colon K \to K^n$ given by $f(\lambda) = (\lambda b_1,\ldots,\lambda b_n)$.
\end{proof}

We now prove the main result of this section.

\begin{theorem} \label{thm:infin3}
Suppose that $K$ is large and not separably closed.
Let $\Kk$ be a $|K|^+$-saturated elementary extension of $K$.
Then there is a non-empty $K$-infinitesimal set $E$ of non-zero elements of $\Kk$ such that the inclusions $K \to \Frac(K[E]^\mathrm{h}) \to \Kk$ are elementary.
\end{theorem}

\begin{proof}
Let $\kappa = |K|$.
We inductively build an increasing chain $(K_i)_{i \in \Nn}$ of elementary subfields of $\Kk$, each of cardinality $\kappa$.
Let $K_0 = K$.
Given $K_i$, we fix a non-zero $K_i$-infinitesimal $\varepsilon_i \in \Kk$ and let $K_{i + 1}$ be an elementary subfield of $\Kk$ of cardinality $\kappa$ containing $\varepsilon_i K^\times_i$.
Note that $K_{i + 1}$ contains $K_i$ as every non-zero element of $K_i$ is a quotient of elements of $\varepsilon_i K^\times_i$.
By Corollary~\ref{cor:infin.5}, $\varepsilon_i K^\times_i$ is a $K_i$-infinitesimal set for each $i$.
Let $E_n = \bigcup_{i = 0}^{n} \varepsilon_i K^\times_i$ for each $n$ and set $E = \bigcup_{n \in \Nn} E_n$, so each $E_n$ is contained in $K_{n + 1}$.
Then $E$ is a $K$-infinitesimal set by Lemma~\ref{lem:infin1}.
Let $K_\infty =\bigcup_{i \in \Nn} K_i$.
We have $K \preceq K_\infty \preceq \Kk$.
We show that $K_\infty = \Frac(K[E]^\mathrm{h})$.
Every element of $K^\times_i$ is a quotient of elements of $\varepsilon_i K^\times_i \subseteq E \subseteq K[E]^\mathrm{h}$, hence each $K_i$ is contained in $\Frac(K[E]^\mathrm{h})$.
Furthermore each $K_{i +1}$ contains the relative algebraic closure of $K \cup E_i$ in $\Kk$, hence each $K_{i + 1}$ contains $\Frac(K[E_i]^\mathrm{h})$.
It follows that
\[K_\infty = \bigcup_{i \in \Nn} K_i = \bigcup_{i \in \Nn} \Frac(K[E_i]^\mathrm{h}) = \Frac(K[E]^\mathrm{h}).
\qedhere\]
\end{proof}

\begin{corollary}\label{cor:thmA}
Suppose that $K$ is large.
Then there is a henselian local domain $R$ extending $K$ with residue field $K$ such that $\Frac(R)$ is an elementary extension of $K$.
\end{corollary}

We thank Philip Dittmann for providing an improved proof of the separably closed case.

\begin{proof}
If $K$ is not separably closed, then we can take $R$ to be $K[E]^\mathrm{h}$ as in Theorem~\ref{thm:infin3}.
Suppose that $K$ is separably closed.
Let $K^*$ be a non-trivial elementary extension of $K$ and fix $t \in K^* \setminus K$, so $t$ is transcendental over $K$.
Let $v$ be the Gauss valuation on $K(t)$ and extend $v$ to a valuation $v^*$ on $K^*$.
Note that $v$ is henselian as any valuation on a separably closed field is henselian.
Furthermore $v$ is trivial on $K$, so $K$ embeds into the residue field of $v$.
We consider $K$ to be a subfield of the residue field of $v$.
Let $R$ be the set of elements of the valuation ring with residue in $K$.
Then $R$ is henselian with residue field $K$.
Finally $\Frac(R) = K$ as every element of $K$ is a quotient of elements with residue $0$.
\end{proof}




We discuss some classical examples.
First, let $K$ be a local field and $\Kk$ be an arbitrary proper elementary extension of $K$.
Let $\st$ be the standard part map $R \to K$, where $R$ is the set of finite elements of $\Kk$.
Then $R$ is a Henselian valuation ring with fraction field $\Kk$ and residue field $K$ and $\st$ is the residue map.
The residue map to $K$ is the map that takes each finite element of $\Kk$ to the nearest element of $K$.
Second, suppose that $K$ is either algebraically, real, or $p$-adically closed.
Then the field $K(\!(t^\Qq)\!)$ of Hahn series over $K$ is algebraically, real, or $p$-adically closed and so the inclusion $K \to K(\!(t^\Qq)\!)$ is elementary by model completeness of the theory of algebraically, real, or $p$-adically closed fields, respectively.
Furthermore $K(\!(t^\Qq)\!)$ has a natural henselian valuation with residue field $K$.
Again, the residue map of this valuation is a ``standard part map" that takes each ``finite" element of $K(\!(t^\Qq)\!)$ to the nearest element of $K$.
The strengthening of Theorem~A in this remark shows more generally that any large field $K$ admits an elementary extension with a ``standard part map" taking values in $K$; however the ``standard part map" is no longer canonical.

\section{The remaining comparison theorems} \label{sec: comparison}

We prove Theorems~C.2 and~C.3 and other results related to comparing the \'etale-open topology and the finite-closed topology.

\subsection{E- and F-sets over algebraically, real, and $p$-adically closed fields}\label{section:rcr-pcf}
Fact~\ref{fact:E-set2} below is \cite[Lemma~5.2]{firstpaper}.

\begin{fact}\label{fact:E-set2}
The preimage of an E-subset of $V(K)$ under the map $W(K) \to V(K)$ induced by a morphism $W \to V$ is an E-subset of $W(K)$.
\end{fact}

We let $R_n$ be the set of $n$th powers in $K$ and $P_n = R_n \setminus \{0\}$.  A \textbf{positive combination} of members of a family $\C$ of sets is a finite union of finite intersections of members of $\C$.

\begin{proposition}\label{prop:acf-rcf-pcf}
\label{prop:real-p-adic}
Suppose that $K$ is algebraically, real, or $p$-adically closed.
Let $X$ be a subset of $K^m$, and equip $K^m$ with the Zariski, order, or $p$-adic topology, respectively.
Let $f$ range over $K[x_1,\ldots,x_m]$.
Then the following are equivalent:
\begin{enumerate}
\item $X$ is an E-set,
\item $X$ is open and definable,
\item $X$ is a positive combination of sets of the following form:
\begin{itemize}
\item $\{ a \in K^m : f(a) \ne 0\}$, when $K$ is algebraically closed.
\item $\{ a \in K^m : f(a) > 0 \}$, when $K$ is real closed.
\item $\{ a \in K^m : f(a) \in P_n \}$, when $K$ is $p$-adically closed.
\end{itemize}
\end{enumerate}
Furthermore the following are equivalent:
\begin{enumerate}
\setcounter{enumi}{3}  
\item $X$ is an F-set,
\item $X$ is closed and definable,
\item $X$ is a positive combination of sets of the following form:
\begin{itemize}
\item $\{ a \in K^m : f(a) = 0 \}$, when $K$ is algebraically closed.
\item $\{ a \in K^m : f(a) \ge 0 \}$, when $K$ is real closed.
\item $\{ a \in K^m : f(a) \in R_n \}$, when $K$ is $p$-adically closed.
\end{itemize}
\end{enumerate}
\end{proposition}

Note that it follows that if $K$ is algebraically, real, or $p$-adically closed then the complement of an E-subset of $K^m$ is an F-set.
We generalize this to bounded fields in Lemma~\ref{lem: Checkingforstandardetale} below.

\begin{proof}
By definition E-sets and F-sets are definable.
Recall that the $\cE_K$-topology agrees with the Zariski, order, or valuation topology when $K$ is algebraically closed, real closed, or $p$-adically closed, respectively.
Hence every E-set is open and by Theorem~\ref{thm:etale-refine} every F-set is closed.
Thus (1) implies (2) and (4) implies (5).
We next check that (2) implies (3) and (5) implies (6).
The algebraically closed case follows by basic properties of the Zariski topology.
The real closed case is \cite[Thm.~2.7.2]{real-algebraic-geometry}.
The $p$-adically closed case is proven in \cite{robinson-edmund}.

\medskip
It remains to show that (3) implies (1) and (6) implies (4).
Suppose that $K$ is algebraically closed.
If $X$ is Zariski open then $X=V(K)$ for an open subvariety $V$ of $\Aa^m$ and the inclusion $V\to\Aa^m$ is \'etale, hence $X$ is an E-set.
If $X$ is Zariski closed then $X=V(K)$ for a closed subvariety $V$ of $\Aa^m$ and the inclusion $V\to\Aa^m$ is finite, hence $X$ is an F-set.

\medskip
Suppose that $K$ is real closed.
We show that (3) implies (1) in this case.
E-sets are closed under positive combinations by Fact~\ref{fact:E-set}.
Hence we may suppose $X = \{ a \in K^m : f(a) > 0 \}$.  Note that $P_2 = \{a \in K : a > 0\}$ is an E-set, because the squaring map is \'etale away from zero.
Then $X = f^{-1}(P_2)$ is an E-set by Fact~\ref{fact:E-set2}.
The proof that (6) implies (4) is similar, using Fact~\ref{fc-lemma}(1) and the fact that the squaring map $\Aa^1 \to \Aa^1$ is a finite morphism, and so $R_2 = \{a \in K : a \ge 0\}$ is an F-set.

\medskip
The case when $K$ is $p$-adically closed follows in the same way as the real closed case by replacing $P_2$ and $R_2$ with $P_n$ and $R_n$.
\end{proof}

\subsection{Bounded fields}\label{sec:perfect-bdd} \label{sec: boundedfields}

We now study when the finite-closed topology refines the \'etale-open topology.

\begin{lemma}\label{confusing-v2}
Let $y = (y_1,\ldots,y_n)$ be a tuple of variables and let $A = K[y]$.
Suppose $x$ is a single variable, and  $f(x,y), g(x,y)$ are two polynomials in $A[x]= K[x,y]$, with
$f(x,y)$ monic in $x$.
For every finite Galois extension $L/K$ there is a finite morphism $W \to \Aa^n$ such that 
the following are equivalent for every $b \in K^n$:
\begin{itemize}[leftmargin=*]
\item $b$ is in the image of $W(K) \to K^n$
\item The polynomial $f(x,b) \in K[x]$ splits over $L$, and
every simple root in $K$ is a root of $g(x,b)$.
\end{itemize}
\end{lemma}

We rephrase Lemma~\ref{confusing-v1} more explicitly in terms of polynomials.
Given an ideal $I$ in $K[x_1,\ldots,x_n]$ we let $V(I)$ be the set of $a \in K^n$ such that every $f\in I$ vanishes on $a$.

\begin{lemma6.2alt}\label{confusing-v1}
Let $y = (y_1,\ldots,y_n)$ be an $n$-tuple of variables, and $x$ be a single variable. 
Let $f(x,y)$ and $g(x,y)$ be two polynomials in $K[x,y]$ with $f(x,y)$ monic in $x$, so
\begin{align*}
f(x,y) &= x^m + a_{m-1}(y)x^{m-1} + \cdots + a_1(y)x + a_0(y)
\\ g(x,y) &= b_\ell(y)x^\ell + b_{\ell-1}(y)x^{\ell-1} + \cdots +
b_1(y)x + b_0(y).
\end{align*}
where each $a_i$ and $b_j$ is in $K[y]$.
Let $L/K$ be a finite Galois extension of degree $d$, and let $z = (z_{1,1},\ldots,z_{m,d})$ be an $md$-tuple of variables. 
Then there is an ideal $I\subseteq K[y,z]$ so that $K[y] \to K[y,z]/I$ is a  finite homomorphism, and the following are  equivalent for any $b \in K^n$:
\begin{itemize}[leftmargin=*]
\item There is $c \in K^{md}$ such that $(b,c) \in V(I)$.
\item $f(x,b) \in K[x]$ splits over $L$, and every simple root in $K$ is also a root of $g(x,b)$.
\end{itemize}
\end{lemma6.2alt}

\begin{proof}[Proof of Lemma \ref*{confusing-v2}$'$ and thus also Lemma \ref{confusing-v2}]
Applying the primitive element theorem, take $\alpha \in L$ such that $L = K(\alpha)$.
Then $\{1, \alpha, \ldots, \alpha^{d-1}\}$ is a $K$-basis of $L$.
For $1 \le i \le m$, let
\begin{equation*}
r_i(z) = z_{i,1} + \alpha z_{i,2} + \cdots + \alpha^{d-1}z_{i,d}
\in L[z].
\end{equation*}
In the ring $L[x,y,z]$, we can expand
\begin{equation*}
f(x,y) - \prod_{i=1}^m[x-r_i(z)] =
\sum_{i=0}^{m-1}\sum_{j=0}^{d-1} \alpha^j x^i h_{i,j}(y,z)
\end{equation*}
for some polynomials $h_{i,j}(y,z) \in K[y,z]$, because the left side is a difference of two degree $m$ monic polynomials over $L[y,z]$, and because $\{1,\alpha,\ldots,\alpha^{d-1}\}$ is a basis of $L[y,z]$ over $K[y,z]$.
Similarly, for each $i = 1,\ldots,m$ we can expand
\begin{equation*}
g(r_i(z),y) \prod_{j \ne i} \prod_{\sigma \in \Gal(L/K)} (\sigma(r_i(z)) - r_j(z)) = \sum_{j=0}^{d-1} \alpha^j q_{i,j}(y,z)
\end{equation*}
for some $q_{i,j}(y,z) \in K[y,z]$.
Here, $\Gal(L/K)$ acts on $L[y,z]$ coefficientwise, i.e., fixing the tuples $y$ and $z$.

\medskip
Let $I_0$ be the ideal in $K[y,z]$ generated by the $h_{i,j}$ and $q_{k,l}$.
\begin{claim}
The map $K[y] \to K[y,z]/I_0$ is a finite homomorphism.
\end{claim}
\begin{claimproof}
Let $I$ be the ideal of $K[y,z]$ generated by the $h_{i,j}$.
It suffices to show $K[y] \to K[y,z]/I$ is finite.
Let $I_L = I \cdot L[y,z]$, i.e., the ideal in $L[y,z]$ generated by the $h_{i,j}$.  
Consider the following diagram.
\begin{equation*}
\xymatrix{ K[y] \ar[r] \ar[d] & K[y,z]/I \ar[d] \\ L[y] \ar[r] & L[y,z]/I_L}
\end{equation*}
The vertical homomorphisms are both of the form $B \to B \otimes_K
L$, so both are finite and injective.
We will see shortly that $L[y] \to L[y,z]/I_L$ is finite.
This implies that the composition $K[y] \to L[y,z]/I_L$ is finite.
Since $K[y]$ is noetherian, the submodule $K[y,z]/I$ is also a finitely generated module, proving the claim.

\medskip
It remains to show that $L[y] \to L[y,z]/I_L$ is finite.
Let $A = L[y,z]/I_L$.
In $A[x]$ we have
\begin{equation*}
f(x,y) - \prod_{i=1}^m(x - r_i(z)) = \sum_{i=0}^{m-1}\sum_{j=0}^{d-1} \alpha^j x^i h_{i,j}(y,z) = 0.
\end{equation*}
Then $f(x,y)$ factors as $\prod_{i=1}^m(x-r_i(z))$, so we have $f(r_i(z)) = 0$ in $A$.
Now $f$ is fixed by $\Gal(L/K)$ as $f \in K[x,y]$, hence we have
\begin{equation*}
A \models f(\sigma(r_i(z))) = \sigma(f(r_i(z))) = 0 \quad \text{for any } \sigma \in \Gal(L/K).
\end{equation*}
As $f$ is monic, it follows that each  $\sigma(r_i(z)) \in A$ is integral over $L[y]$.
Note that
\begin{equation*}
\sigma(r_i(z)) = z_{i,1} + \sigma(\alpha)z_{i,2} + \cdots +
\sigma(\alpha)^{d-1}z_{i,d}.
\end{equation*}
Then the matrix of $\sigma(r_i(z))$'s is obtained from the matrix of $z_{i,j}$'s by multiplication by the Vandermonde matrix of $\sigma(\alpha)^j$'s.
This matrix is invertible, so the $z_{i,j}$'s are $L$-linear combinations of the $\sigma(r_i(z))$'s.
As an $L[y]$-algebra, $A$ is generated by the $z_{i,j}$'s, and therefore also by the $\sigma(r_i(z))$'s.
Each generator is integral over $L[y]$, hence $L[y] \to A$ is finite.
\end{claimproof}
We now fix $b \in K^n$.
We claim that the following are  equivalent:
\begin{enumerate}
\item $f(x,b) \in K[x]$ splits over $L$ and every simple root in $K$ is also a root of $g(x,b)$.
\item There are $s_1,\ldots,s_m \in L$ such that
\begin{gather*}
f(x,b) = \prod_{i=1}^m (x - s_i) \text{\quad in  } L[x] \tag{$\ast$} \\
g(s_i,b) \prod_{j \ne i} \prod_{\sigma \in \Gal(L/K)} (\sigma(s_i) - s_j) = 0 \text{\quad for all\quad} i = 1,\ldots,m. \tag{$\dag_i$}
\end{gather*}
\item There is $c \in K^{md}$ such that
\begin{gather*}
f(x,b) = \prod_{i=1}^m ( x - r_i(c)) \text{\quad in  } L[x]
\\ g(r_i(c),b) \prod_{j \ne i} \prod_{\sigma \in \Gal(L/K)}   (\sigma(r_i(c)) - r_j(c)) = 0
\end{gather*}
\item There is $c \in K^{md}$ such that all the $h_{i,j}(b,c)$ and $q_{i,j}(b,c)$ vanish.
\item There is $c \in K^{md}$ such that $(b,c) \in V(I_0)$.
\end{enumerate}
$(1)\implies(2)$: Let $s_1,\ldots,s_m \in L$ be the roots of $f(x,b)$ counted with multiplicities.
Then ($\ast$) holds.
For each $i$, at least one of the following holds:
\begin{itemize}
\item $s_i$ is a root of $g(x,b)$, so $g(s_i,b) = 0$.
\item $s_i$ is a non-simple root.
Then $s_i - s_j = 0$ for some $j \ne i$, so $\sigma(s_i) - s_j = 0$ for $\sigma = \id$.
\item $s_i$ is not in $K$.
Then $\sigma(s_i)\ne s_i$ for some $\sigma \in \Gal(L/K)$.
Since $\sigma$ permutes the roots of $f(x,b)$,
$\sigma(s_i) - s_j = 0$ for some $j \ne i$.
\end{itemize}
Either way, $(\dag_i)$ holds.

\medskip
$(2)\implies(1)$: Equation ($\ast$) shows that $f(x,b)$ splits over $L$.
Suppose for the sake of contradiction that $s_i$ is a simple root in $K$ and $g(s_i,b) \ne 0$.
For any $\sigma \in \Gal(L/K)$ and $j \ne i$ we have $\sigma(s_i) = s_i \ne s_j$.
Therefore $(\dag_i)$ fails.

\medskip
$(2)\iff(3)$: Since $\{1,\alpha,\ldots,\alpha^{d-1}\}$ is a basis of $L$ over $K$ and $$r_i(z) = z_{i,1} + \alpha z_{i,2} + \cdots + \alpha^{d-1} z_{i,d},$$ the equivalence is clear.

\medskip
$(3)\iff(4)$: clear by choice of the $h$'s and $q$'s.

\medskip
$(4)\iff(5)$: clear by choice of $I_0$.\qedhere
\end{proof}

\begin{lemma}\label{lem:bnddkey}
Assume the conditions of Lemma~\ref{confusing-v2} and suppose in addition that $K$ is bounded. 
Then there is a finite morphism  $W \to \Aa^n$ so that the following are equivalent for any $b \in K^n$:
\begin{itemize}
\item $b$ is not in the image of the induced map $W(K) \to K^n$.
\item The polynomial $f(x,b)$ has a simple root in $K$ 
that is not also a root of $g(x,b)$.
\end{itemize}
\end{lemma}

\begin{proof}
Write $f$ and $g$ as
\begin{align*}
f(x,y) &= x^m + a_{d-1}(y)x^{m-1} + \cdots + a_1(y)x + a_0(y) 
\\ g(x,y) &= b_\ell(y)x^\ell + b_{\ell-1}(y)x^{\ell-1} + \cdots + b_1(y)x + b_0(y).
\end{align*}
In particular, $m$ is the degree of $f$ as a monic polynomial in $m$.
By boundedness, there is a finite Galois extension $L/K$ containing an isomorphic copy of every finite separable extension of $K$ of degree $\le m$.
If $K$ is perfect then the polynomial $f(x,b)$ splits over $L$ for any $b \in K^n$, and we are done by Lemma~\ref{confusing-v2}.

\medskip
Suppose that $K$ is imperfect of characteristic $p$.
Fix $k$ such that $p^k > m$.
Let
\begin{align*}
f_2(x,y) &= x^m + a_{d-1}(y)^{p^k} x^{m-1} + \cdots + a_1(y)^{p^k} x + a_0(y)^{p^k}
\\ g_2(x,y) &= b_\ell(y)^{p^k} x^\ell + b_{\ell-1}(y)^{p^k} x^{\ell-1} + \cdots + b_1(y)^{p^k} x + b_0(y)^{p^k}.
\end{align*}
For any $b \in K^n$, the roots of $f_2(x,b)$ in $\kalg$ are $p^k$th powers of the roots of $f(x,b)$.
Similarly, the roots of $g_2(x,b)$ are the $p^k$th powers of the roots of $g(x,b)$.




\begin{Claim*}
If $b \in K^n$ and $a \in \kalg$ satisfies $f_2(a,b) = 0$, then $a \in L$.
Equivalently $f_2(x,b)$ splits over $L$ for any $b \in K^n$.
\end{Claim*}

\begin{claimproof}
Write $a$ as $a_0^{p^k}$ for some root $a_0$ of $f(x,b)$.
If $P(x)$ is the minimal polynomial of $a_0$ over $K^{\sep}$, then $P(x)$ divides $f(x,b)$, so $\deg P \le \deg f(x,b) = m$.
Also, $P(x)$ has the form $x^{p^j} - c$ for some $j$ and $c \in K^{\sep}$.  Then $p^j = \deg P(x) \le m < p^k$, so $j < k$.
Then $c = a_0^{p^j}$ and so $a = a_0^{p^k} = c^{p^{k-j}} \in K^{\sep}$.
At the same time, $a$ is a root of $f_2(x,b)$, which has degree $m$, so $[K(a) : K] \le m$.
Then $a \in L$ by choice of $L$ and separability of $K(a)/K$.
\end{claimproof}
  
By Lemma~\ref{confusing-v2} there is a finite morphism $W \to \Aa^n$ such that the following are equivalent for any $b \in K^n$:
\begin{itemize}
\item $b$ is in the image of $W(K) \to \Aa^n(K)$.
\item Every simple root of $f_2(x,b)$ in $K$ is a root of $g_2(x,b)$ (and $f_2(x,b)$ splits over $L$, but this is automatic by the Claim).
\end{itemize}
Let $\perf$ be the maximal purely inseparable extension of $K$.
Then the following are equivalent for any $b \in K^n$:
\begin{enumerate}
\item The polynomial $f(x,b)$ has a simple root in $K$ that is not also a root of $g(x,b)$.
\item The polynomial $f(x,b)$ has a simple root in $\perf$ that is not also a root of $g(x,b)$.
\item The polynomial $f_2(x,b)$ has a simple root in $\perf$ that is not also a root of $g_2(x,b)$.
\item The polynomial $f_2(x,b)$ has a simple root in $K$ that is not also a root of $g_2(x,b)$.
\item $b$ isn't in the image of $W(K) \to \Aa^n(K)$.
\end{enumerate}
The equivalence of (1) and (2) holds because simple roots of polynomials over $K$ always lie in $K^{\sep}$, and $K^{\sep} \cap \perf = K$.
The equivalence of (3) and (4) holds similarly.
The equivalence of (2) and (3) holds because the map $x \mapsto x^{p^k}$ maps the roots of $f(x,b)$ and $g(x,b)$ bijectively to the roots of $f_2(x,b)$ and $g_2(x,b)$.
Finally, the equivalence of (4) and (5) holds by choice of $W$, as discussed above.
\end{proof}

The previous proof is implicitly using the iterated relative Frobenius morphism; see \cite[Definition~0CC9]{stacks-project} or \cite[\S~1.2]{secondpaper}.

\begin{lemma} \label{lem: Checkingforstandardetale}
Let $K$ be bounded.
Then the complement of any E-subset of $K^n$ is an F-set.
\end{lemma}

\begin{proof}
Let $f\colon V \to \Aa^n$ be \'etale.
We show that the complement of $f(V(K))$ in $K^n$ is an F-subset of $K^n$.
By Fact~\ref{fac: sdetalevsetale} there is a cover $V_1,\ldots,V_d$ of $V$ by affine open subvarieties such that each $V_i \to \Aa^n$ is standard \'etale.
The complement of $f(V(K))$ is the union of the complements of the $f(V_i(K))$.
F-subsets of $K^n$ are closed under finite unions by Lemma~\ref{fc-lemma}(3), so we may suppose that $f$ is standard \'etale.
Hence we may suppose that $V$ is the subvariety of $\Aa^n \times \Aa^1$ given by $g = 0 \ne h$ for some $g,h \in K[x_1,\ldots,x_n,y]$ such that $g$ is monic in $y$ and $\der g/\der y$ does not vanish on $V$, and $f$ is the projection $V \to \Aa^n$.
Then the fiber of $V(\kalg) \to \Aa^n(\kalg)$ over any $b \in \Aa^n(\kalg)$ is exactly the set of simple roots of $g(b,x)$ which are not also roots of $h(b,x)$.
In particular, if $b \in K^n$, then $b$ lies in $f(V(K))$ if and only if there
is a simple root of $g(b,x)$ in $K$ that is not  a root of $h(b,x)$.
Apply Lemma~\ref{lem:bnddkey}.
\end{proof}

Lemma~\ref{lem: Checkingforstandardetale} shows that the $\cF_K$-topology on $K^n$ refines the $\cE_K$-topology for any $n \ge 1$ when $K$ is bounded.
Theorem~\ref{66}, which is Theorem~C.2 from the introduction, follows by Fact~\ref{fact:refine}.

\begin{theorem} \label{66}
The $\cF_k$-topology refines the $\cE_K$-topology when $K$ is bounded.
\end{theorem}

Corollary~\ref{last-cor} follows by Theorem~\ref{66}, Theorem~\ref{thm:etale-refine}, and Lemma~\ref{lem:system-2}.

\begin{corollary} \label{last-cor}
Suppose that $K$ is perfect and bounded.
Then the $\cF_K$- and $\cE_K$-topologies agree.
Hence $V(K) \to W(K)$ is a closed map in the $\cE_K$-topology when $V \to W$ is finite.
\end{corollary}

\subsection{The t-henselian case}\label{section:t-henselian}
Let $\uptau$ be a Hausdorff field topology on $K$.
Then $X\subseteq K$ is \textbf{$\uptau$-bounded} if every neighborhood of $0$ contains $aX$ for some $a \in K^\times$.
Furthermore $\uptau$ is a V-topology if $(K\setminus U)^{-1}$ is $\uptau$-bounded for any neighborhood $U$ of $0$.
By \cite[Thm.~B.1]{EP-value} $\uptau$ is a V-topology if and only if $\uptau$ is induced by an absolute value or valuation on $K$.
Furthermore $\uptau$ is t-henselian if $\uptau$ is a V-topology and for any $n \ge 1$ there is a $\uptau$-open neighborhood $U$ of $0$ such that $x^{n+2} + x^{n + 1} + c_n x^n + \cdots + c_1 x +  c_0$ has a root in $K$ for any $c_0,\ldots,c_n \in U$.
Equivalently: a V-topology $\uptau$ on $K$ is t-henselian if and only if $V(K) \to W(K)$ is open for any \'etale $V \to W$~\cite[Props~8.3, 8.6]{field-top-2}.
Finally $K$ is \textbf{t-henselian} if $K$ admits a t-henselian field topology.
The main examples of t-henselian topologies are the topology induced by a non-trivial henselian valuation on $K$ and the order topology on a real closed field.
A non-separably closed field admits at most one t-henselian field topology.
See \cite[\S7]{Prestel1978} for this and other background results.

\medskip
We now prove Theorem~C.3 from the introduction.

\begin{theorem}\label{thm:new-t-hensel}
Suppose that $\uptau$ is a t-henselian topology on a perfect field $K$.
\begin{enumerate}[leftmargin=*]
\item The $\cE_K$-topology agrees with the $\cF_K$-topology.
\item If $K$ is not algebraically closed, then the $\cE_K$-topology and the $\cF_K$-topology agree with the system $\cT_\uptau$ of topologies induced by $\uptau$.
\end{enumerate}
\end{theorem}

\begin{proof}
The algebraically closed case follows as the $\cE_K$- and $\cF_K$-topologies both agree with the Zariski topology when $K$ is algebraically closed.
Suppose that $K$ is not algebraically closed.
Then $\cT_\uptau$ agrees with the $\cE_K$-topology by Fact~\ref{fact:old-EO}(\ref{old:hnsl}), and the $\cE_K$-topology refines the $\cF_K$-topology by Theorem~\ref{thm:etale-refine}.
It remains to show that the $\cF_K$-topology refines $\cT_\uptau$.
By Fact~\ref{fact:old-tau-compare} below, it suffices to produce a non-empty $\cF_K$-open subset of $K$ which is not $\cE_K$-dense in $K$.
By Proposition~\ref{prop:hd} we have $\varnothing \ne U \subseteq F \subsetneq K$ for an $\cE_K$-open subset $U$ and an $\cF_K$-closed subset $F$.
Hence $K \setminus F$ is $\cF_K$-open and not $\cE_K$-dense.
\end{proof}

\begin{fact}\label{fact:old-tau-compare}\textup{(}\cite[Lemma~6.9]{firstpaper}\textup{)}
Let $\uptau$ be a V-topology on $K$, let $\cT_\uptau$ be the system of topologies induced by $\uptau$, and let $\cT$ be an arbitrary system of topologies over $K$.
Suppose that some non-empty $\cT$-open subset of $K$ is not $\uptau$-dense in $K$.
Then $\cT$ refines $\cT_\uptau$.
\end{fact}

Suppose that $K$ is characteristic zero.
Theorem~\ref{thm:new-t-hensel} shows that the \'etale-open and finite-closed topologies agree over $K(\!(t)\!)$.
This does not generalize to $K(\!(t_1,\ldots,t_n)\!)$ when $n \ge 2$.
By \cite[Thm.~7.2]{Weissauer} $K(\!(t_1,\ldots,t_n)\!)$ is Hilbertian, so it follows by Proposition~\ref{prop:hilbertian} below that the \'etale-open and finite-closed topologies do not agree.

\subsection{A topological application of the comparison theorem}\label{section:galois}

\begin{proposition}\label{prop:galois}
Suppose that $L/K$ is a Galois extension, $K$  is perfect, $L$ is not separably closed, and $K, L$ are both bounded.
Then the inclusion $V(K) \to V(L)$ is a closed embedding from the $\cE_K$-topology on $V(K)$ to the $\cE_L$-topology on $V(L)$.
\end{proposition}

Note that a finite separable extension of a bounded field is bounded, so this applies in the case when $K$ is perfect and bounded and $L$ is a finite non-separably closed extension of $K$.
We will give the proof shortly.  Our proof of Proposition~\ref{prop:galois} goes through provided that the \'etale-open and finite-closed topologies agree over $K$ and $L$.
We first recall some background.

\begin{fact}\label{fact:rel}
Let $L/K$ be an algebraic extension and $V$ be a $K$-variety.
\begin{enumerate}[leftmargin=*]
\item $V(K) \cap O$ is $\cE_K$-open for any $\cE_L$-open subset $O$ of $V(L)$.
\item If $L/K$ is Galois and $L$ is not separably closed then $V(K)$ is an $\cE_L$-closed subset of $V(L)$.
\end{enumerate}
\end{fact}

Part (1) is \cite[Thm.~5.8]{firstpaper} and (2) is \cite[Cor.~2.8]{secondpaper}.
We now prove Proposition~\ref{prop:galois}.

\begin{proof}
By Fact~\ref{fact:rel} it is enough to show that $V(K) \to V(L)$ is a closed map from the $\cE_K$-topology to the $\cE_L$-topology.
As the $\cE_K$- and $\cF_K$-topologies agree it suffices to fix a finite morphism $f\colon W \to V$  and show that $f(W(K))$ is $\cE_L$-closed in $V(L)$.
Let $f_L\colon V_L \to W_L$ be the base change of $f$.
Then $f_L$ is finite and hence gives a closed map in the $\cE_L$-topology by Corollary~\ref{last-cor}.
By Fact~\ref{fact:rel} $W(K)$ is $\cE_L$-closed in $W_L(L)$, hence $f_L(W(K)) = f(W(K))$ is also $\cE_L$-closed.
\end{proof}

We describe an example.
Let $F$ be a bounded $\prc$ field which is neither real closed nor $\pac$.
Then $F$ admits exactly $n$ distinct field orders $<_1,\ldots,<_n$ for some $n \ge 1$~\cite[Remark~3.2]{Samaria-2017}.
For example, we could suppose that $(F,<_1,\ldots,<_n)$ satisfies the model  companion of the theory of a field equipped with $n$ field orders for some $n \ge 2$.
Then $L = F[\sqrt{-1}]$ is $\pac$ and not algebraically closed and $L/F$ is Galois.
Hence the $\cE_F$-topology on $F$ agrees with the topology induced on $F$ by the $\cE_L$-topology.
This is surprising since\ldots

\begin{itemize}[leftmargin=*]
\item The $\cE_F$-topology refines each $<_i$-topology by Fact~\ref{fact:old-EO}(\ref{old:ordr}).
\item The $\cE_L$-topology should be hostile to orders, since $L$ is an $\mathrm{NSOP}$ structure \cite{Chatzidakis}.
\end{itemize}
This behavior is also markedly different from the behavior of $\Cc/\Rr$: $\cE_{\Rr}$ is the order topology and $\cE_{\Cc}$ is the Zariski topology, so $\cE_{\Rr}$ is \emph{not} the restriction of $\cE_{\Cc}$.

\medskip
We need to assume that $L$ is not separably closed in Proposition~\ref{prop:galois}: if $L$ is a separably closed Galois extension of $K$ then $K$ is not separably closed, hence the $\cE_K$-topology is Hausdorff, whereas the restriction of the $\cE_L$-topology to $K$ is the cofinite topology.
We also need to assume that $L/K$ is Galois as there are examples of non-separably closed algebraic extensions $L/K$ of bounded fields such that the $\cE_K$-topology does not agree with the restriction of the $\cE_L$-topology.
For example, let $F$ be as above.
Let $L$ be the real closure of $F$ with respect to $<_1$.
Then $L$ is real closed and hence the $\cE_L$-topology agrees with the order topology by Fact~\ref{fact:old-EO}(\ref{old:ordr}).
Then the restriction of the $\cE_L$-topology to $F$ is the $<_1$-order topology (because $F$ is cofinal in $L$).
However, the $\cE_F$-topology refines the $<_i$ order topology for each $i = 1,\ldots,n$ by Fact~\ref{fact:old-EO}(\ref{old:ordr}) and any nonempty $<_i$-open subset of $F$ is $<_j$-dense when $i \ne j$~\cite[Prop.~1.6]{prestel-prc}.

\section{Examples and counterexamples}
We present examples which show that our results are sharp in various ways.

\subsection{A bounded $\pac$ field over which the finite-closed topology is discrete}\label{section:Esharp}
It is natural to ask whether any infinite field can have discrete $\cF_K$-topology.
We give an example in this section, at the same time proving Theorem~D (answering Lampe's question).

\begin{fact}\label{fact:sys-dis}
The following are equivalent for any system $\cT$ of topologies over $K$.
\begin{enumerate}[leftmargin=*]
\item $\cT$ is the discrete system of topologies.
\item The $\cT$-topology on $K$ is discrete.
\item There is a nonempty $\cT$-open subset of $K$ of cardinality $< |K|$.
\end{enumerate}
\end{fact}

\begin{proof}
The equivalence of (1) and (2) is \cite[Prop.~4.5]{firstpaper}.
It is clear that (2) implies (3).
Suppose that $U$ is nonempty, $\cT$-open, and $|U| < |K|$.
After possibly translating suppose that $0 \in U$.
As $|U| < |K|$, some $\beta \in K^\times$ is not a quotient of elements of $U$.
Hence $U \cap \beta U = \{0\}$, and so $\{0\}$ is $\cT$-open.
Translations $K \to K$ are $\cT$-homeomorphisms, so $\{c\}$ is $\cT$-open for any $c \in K$, and $\cT$ is discrete.
\end{proof}

\begin{remark} \label{rem:lampe-discrete}
Let $f \colon K \to K$ be a polynomial map with $0 < |K \setminus f(K)| < |K|$.
Then the $\cF_K$-topology is discrete.
To see this, note that $f \colon \Aa^1 \to \Aa^1$ is a finite morphism by Fact~\ref{fact:basic-f}(\ref{poly:fini}), so $f(K)$ is an F-set and $K \setminus f(K)$ is an open set in the $\cF_K$-topology.
Then some nonempty set of cardinality $<|K|$ is open, and the topology is discrete by Fact~\ref{fact:sys-dis}.
\end{remark}



Recall that $K$ is {\bf existentially closed} in a $K$-algebra $A$ if one of the following equivalent conditions are satisfied:
\begin{enumerate}[leftmargin=*]
\item Whenever $\alpha_1,\ldots,\alpha_n \in K$ and $\varphi(x_1,\ldots,x_n)$ is an existential formula in the language of rings  then $K \models \varphi(\alpha_1,\ldots,\alpha_n)$ if and only if $A \models \varphi(\alpha_1,\ldots,\alpha_n)$.
\item Any finite system of polynomial equations and inequations with coefficients from $K$ and a solution in $A$ has a solution in $K$.
\end{enumerate}
A field extension $L/K$ is {\bf regular} if it is separable and $K$ is algebraically closed in $L$.
A field $K$ is $\pac$ if it is existentially closed in any regular field  extension~\cite[Prop.~11.3.5]{field-arithmetic}.
Given a class $\C$ of fields we say that $K \in \C$ is existentially closed in $\C$ if $K$ is existentially closed in any field in $\C$ which extends $K$.
Fact~\ref{fact:exx} is a special case of a general model-theoretic fact~\cite[Thm.~8.2.1]{Hodges}

\begin{fact}\label{fact:exx}
If $\C$ is a class of fields which is closed under increasing unions then any field in $\C$ has an extension which is existentially closed in $\C$.
\end{fact}

Let $\Ff_4$ be the field with four elements.
Throughout this section we let $f(x) \in \left(\Ff_4(t)\right)[x]$ be $$f(x) = (x^2 + t)(x^3 + t).$$

\begin{proposition}\label{prop:cofin}
Let $\C$ be the class of fields extending $\Ff_4(t)$ which do not contain a square root or cube root of $t$.
Let $K \in \C$ be existentially closed in $\C$.
Then the following holds:
\begin{enumerate}[leftmargin=*]
\item $K$ is bounded, $\pac$, and $f(K) = K^\times$.
\item The $\cF_K$-topology is discrete and the $\cE_K$-topology is not discrete.
\end{enumerate}
Moreover, at least one such $K$ exists.
\end{proposition}

We give the proof over the remainder of the subsection.
Fact~\ref{fact:exx} gives an existentially closed $K$.
Moreover, (2) follows from (1): the $\cF_K$-topology is discrete by Remark~\ref{rem:lampe-discrete} and the $\cE_K$-topology is not discrete as $K$ is $\pac$ and hence large.
It remains to prove (1).

\begin{fact}\label{fact:group}
Let $p$ be a prime.
If $G$ is a finite group which admits an index $p$ subgroup containing every proper subgroup then $G$ is a cyclic group with order a power of $p$
\end{fact}

Fact~\ref{fact:group} is \cite[\S{}5.2, Exercise~14]{a_course_in_the_theory_of_groups}.

\begin{corollary} \label{z3}
Let $K$ be a field and $F/K$ be a Galois extension of degree 3.
Suppose that any proper finite separable extension of $K$ contains $F$.
\begin{enumerate}[leftmargin=*]
\item The Galois group of any finite Galois extension of $K$ is cyclic with order a power of $3$.
\item $K$ is bounded.
In fact $K$ has at most one separable field extension of any given degree up to isomorphism.
\end{enumerate}
\end{corollary}

\begin{proof}
Note that (1) follows from Fact~\ref{fact:group}.
For (2), let $L_1, L_2$ be separable extensions of $K$ of degree $n$.
Let $L_3$ be a finite Galois extension of $K$ containing $L_1$ and $L_2$.
Then $\Gal(L_3/K)$ is isomorphic to $\Zz/3^k\Zz$ for some $k$.
The fields $L_1$ and $L_2$ correspond to index $n$ subgroups of $\Gal(L_3/K)$.
Hence $L_1 = L_2$ as $\Zz/3^k\Zz$ has at most one subgroup of any given index.
\end{proof}

\begin{proof}[Proof of Proposition~\ref{prop:cofin}]
It is clear that $\C$ is closed under regular extensions, hence $K$ is existentially closed in any regular extension, hence $K$ is $\pac$.

\medskip
Let $F = K(\sqrt[3]{t})$.
Then $F/K$ is a degree $3$ Galois extension as $\Ff_4$ contains all cube roots~of~$1$.

\begin{Claim*}
Any proper separable algebraic extension of $K$ contains $F$.
\end{Claim*}

\begin{claimproof}
Let $L/K$ be a proper separable algebraic extension.
Then $K$ is not existentially closed in $L$ and so $L$ contains a square or cube root of $t$.
Now $L$ cannot contain a square root of $t$ as $K(\sqrt{t})/K$ is not separable.
Hence $L$ contains a cube root of $t$ and so $L$ contains $F$.
\end{claimproof}

It follows by Corollary~\ref{z3} that $K$ is bounded and the degree of any finite separable extension of $K$ is a power of $3$.
Consequently, the degree of any separable irreducible polynomial in $K[x]$ is a power of $3$.

\medskip
We now show that $f(K) = K^\times$.  Recall that $f(x) = (x^2 + t)(x^3 + t)$.
Then $0 \notin f(K)$ as $f$ has no roots in $K$ and $t^2 \in f(K)$ as $f(0) = t^2$.
Fix $c \in K \setminus \{0, t^2\}$.
We show that $g(x) = f(x) + c$ has a root in $K$.
Note that
\[
g'(x) = f'(x) = 2x(x^3 + t) + 3x^2(x^2 + t) = x^2(x^2 + t).
\]
Hence the roots of $g'(x)$ are $0$ and $\sqrt{t}$.
We supposed that $c \ne 0 = f(\sqrt{t})$ and $c\ne t^2 = f(0)$, hence $g(\sqrt{t}) \ne 0 \ne g(0)$.
Then $g'(x)$ and $g(x)$ do not share any roots, so $g$ is separable.

\medskip
Therefore $g$ factors as a product of irreducible separable polynomials.
Each factor has degree a power of $3$.
Then at least one factor is linear, since $5 = \deg(g)$ cannot be written as a sum of numbers of the form $3^m$ with $m > 0$.
Thus $g$ has a root in $K$.
\end{proof}

\begin{remark}
There is an algebraic extension $F$ of $\Ff_4(t)$ such that $F$ is bounded $\pac$ and $f(F) = F^\times$ with $f(x)$ as above.
We sketch the proof.  
By the proof of Proposition~\ref{prop:cofin}, it suffices to find a bounded $\pac$ extension $K \supseteq \Ff_4(t)$ such that $f(x)$ has no roots in $K$, and the degree of any finite separable extension of $K$ is a power of 3.  Equip $\Gal(\Ff_4(t))$ with the unique Haar probability measure and let $F^*$ be the fixed field of a random element of $\Gal(\Ff_4(t))$.
Then $F^*$ is $\pac$ and $\Gal(F^*)$ is the profinite completion of $\Zz$ with probability $1$~\cite[Thm.~20.8.2]{field-arithmetic}.  With probability $2/3$, $F^*$ does not contain the cube roots of $t$, and so $[F^*(\sqrt[3]{t}) : F^*] = 3$.  Fix such an $F^*$.  Let $F$ be the Galois extension of $F^*$ with absolute Galois group $\Zz_3$.  Then $F$ is $\pac$ because it's an algebraic extension of a $\pac$ field, and every finite separable extension of $F$ has degree a power of 3.  Since $F$ is separable over $\Ff_4(t)$, it does not contain the square root $\sqrt{t}$.  Finally, $F$ does not contain the cube roots of $t$ because $\Gal(F/F^*) \cong \Zz_2 \times \Zz_5 \times \Zz_7 \times \cdots$ is prime to 3, and $[F^*(\sqrt[3]{t}) : F^*] = 3$.
\end{remark}

Let $h$ range over $K[x]$.
We constructed a large field over which the finite-closed topology is discrete by producing large $K$ and $h$ with $0 < |K \setminus h(K)| < |K|$.
In the other direction, we know that the finite-closed topology over a perfect large field is not discrete, so $K \setminus h(K)$ has cardinality $0$ or $|K|$ when $K$ is perfect and large.
This is a special case of Fact~\ref{fact:finite-image}, due to Bary-Soroker, Geyer, and Jarden~\cite{bsgj},  generalizing work of Kosters~\cite{kosters}.

\begin{fact}
\label{fact:finite-image}
Suppose that $K$ is large and let $L$ be the maximal purely inseparable extension of $K$.
Suppose that $h \colon V \to W$ is a finite morphism of irreducible $K$-varieties and $h_L(V(L))$ does not contain some smooth point of $W(L)$.
Then $W(K) \setminus h_L(V(L))$ has cardinality $|K|$.
\end{fact}

We give a topological proof of Fact~\ref{fact:finite-image}.

\begin{proof}
As $L$ is perfect $h_L(V(L))$ is an $\cE_L$-closed subset of $W(L) = W_L(L)$ by Theorem~\ref{thm:etale-refine}.
By Fact~\ref{fact:kins}(2), $W(K) \cap h_L(V(L))$ is an $\cE_K$-closed subset of $W(K)$.
Apply Fact~\ref{fact:z-dense}.
\end{proof}

\medskip
We describe some other known cases of Lampe's question over large fields.

\begin{fact}\label{fact:lpm-css}
Suppose that $K$ is large, $h \in K[x]$, and $L$ is the maximal purely inseparable extension of $K$.
If one of the following holds then $|K \setminus h(K)| = |K|$.
\begin{enumerate}[leftmargin=*]
\item $h$ is irreducible.
\item $h$ is separable and does not have a root in $K$.
\item $h$ does not have a root in $L$.
\item All roots of $h$ are in $L \setminus K$.
\end{enumerate}
\end{fact}

Here (1) is due to Koenigsmann~\cite[\S\! 6]{open-problems-ample} and (3) and (4) are proven in \cite{bsgj}.
Note that (2) follows from (3) as a separable polynomial cannot have a root in $L \setminus K$.
It is also easy to see that (3) follows from Fact~\ref{fact:finite-image}.

\subsection{Perfect fields over which the $\cE_K$-topology strictly refines the $\cF_K$-topology}\label{section:Dsharp}
For perfect fields, Theorem C shows that $\cE_K$ refines $\cF_K$, but in this subsection we will see examples (such as $\Qq$) where the refinement is strict.

\medskip
Recall that a topological space is {\bf irreducible} if it cannot be written as a union of two proper closed subsets (or equivalently, any two non-empty open subsets intersect).
Fact~\ref{fact:HD} below is \cite[Prop.~4.12]{firstpaper}.

\begin{fact}\label{fact:HD}
The following are equivalent for any system $\cT$ of topologies over $K$.
\begin{enumerate}[leftmargin=*]
\item The $\cT$-topology on $K$ is Hausdorff.
\item The $\cT$-topology on $K$ is not irreducible.
\item The $\cT$-topology on $V(K)$ is Hausdorff for any quasi-projective $V$.
\end{enumerate}
\end{fact}

A subset of $K^n$ is {\bf thin} if it is contained in a set of the form $A \cup f_1(V_1(K)) \cup \cdots \cup f_m(V_m(K))$ where $A\subseteq K^n$ is not Zariski dense and each $f_i$ is a dominant morphism $V_i \to \Aa^n$ of degree $\ge 2$ with $\dim V_i = n$.
We say that $K$ is {\bf Hilbertian} if and only if $K^n$ is not thin for every  $n\ge 1$.
Our notion of Hilbertian agrees with that given in the first edition of \cite{field-arithmetic}, but not with the notion given in later editions.
The two notions agree in characteristic zero.

\medskip
Let $K$ be either a number field or a function field over a characteristic zero field $F$, i.e., $K/F$ is finitely generated and $F$ is relatively algebraically closed in $K$.
Then $K$ is not large as $K$ satisfies a form of the Mordell conjecture~\cite[9.5.1, 9.5.2]{poonen-qpoints}, hence the $\cE_K$-topology is discrete.
Furthermore $K$ is Hilbertian~\cite[Thm.~13.4.2]{field-arithmetic}, so the following shows that the $\cF_K$-topology is not discrete.\footnote{There are also examples of large Hilbertian fields such as $\omega$-free $\pac$ fields~\cite[Cor~27.3.3]{field-arithmetic}.}

\begin{proposition}\label{prop:hilbertian}
If $K$ is Hilbertian then the $\cF_K$-topology on each $K^n$ is irreducible and hence non-Hausdorff and non-discrete.
The $\cF_K$-topology does not agree with the $\cE_K$-topology.
\end{proposition}

\begin{proof}
By Fact~\ref{fact:HD} the $\cF_K$-topology on $K$ is Hausdorff if and only if it is not irreducible.
The second claim follows from the first because separably closed fields are not Hilbertian, and non-separably closed fields have Hausdorff $\cE_K$-topologies by Fact~\ref{fact:old-EO}(\ref{old:sep}).

\medskip
The first claim of Proposition~\ref{prop:hilbertian} follows from the definition of Hilbertianity and Lemma~\ref{lem:hilbertian} below, since a topological space $T$ with closed basis $\mathcal{B}$ is irreducible if and only if $T$ cannot be expressed as a finite union of proper elements of $\mathcal{B}$.
\end{proof}

\begin{lemma}\label{lem:hilbertian}
If $Y \subseteq K^n$ is an F-set, then either $Y = K^n$ or $Y$ is thin.
\end{lemma}

\begin{proof}
Write $Y$ as the $K$-rational image of a finite morphism $\phi \colon V \to \Aa^n$.
Replacing $V$ with its irreducible components, we may assume $V$ is irreducible.
We have $\dim V \le n$ as $\phi$ is finite.
If $\dim V < n$, then the image of $V \to \Aa^n$ is not Zariski dense and is hence thin, so we may suppose $\dim V = n$.
If $\phi$ has degree at least two then $\phi(V(K))$ is thin.
Suppose that $\phi$ has degree $1$.
Then $\phi \colon V \to \Aa^n$ is a birational map with finite fibers.
By the ``original form'' of Zariski's main theorem~\cite[III \S{}9]{red-book}, $\phi$ is an open immersion.
Since $\phi$ is finite, $\phi$ must be an isomorphism.
Therefore $\phi(V(K))= K^n$.
\end{proof}

\begin{proposition}\label{prop:hilbertian1}
Suppose that $K$ is Hilbertian and let $X$ be a subset of $K^n$.
Then $X$ is not $\cF_K$-dense if and only if $X$ is thin.
\end{proposition}

\begin{proof}[Proof sketch]
Lemma~\ref{lem:hilbertian} shows that any subset of $K^n$ which is not $\cF_K$-dense is thin.
Conversely, if $g \colon V \to \Aa^n$ is a morphism of degree $\ge 2$ with $\dim V = n$, then we apply Fact~\ref{fact:zmt} to extend $g$ to a finite morphism $g^*\colon V^* \to \Aa^n$ and note that $g^*(V^*(K))$ is a proper F-subset of $K^n$ containing the thin set $g(V(K))$.
Proposition~\ref{prop:hilbertian1} follows from these observations and an easy argument which we leave to the reader.
\end{proof}

\begin{remark}
Proposition~\ref{prop:hilbertian} suggests that the finite-closed topology is suitable to study non-large Hilbertian fields, especially those that are model-theoretically tame.
Models of $C$XF or $V$XF~\cite{cxf-paper,vxf-paper} give examples of model-theoretically tame non-large Hilbertian  fields. 
It would be interesting to investigate the finite-closed topologies in such examples. 
\end{remark}

\subsection{Large fields that are not fraction fields of proper henselian local domains}\label{section:A-sharp}
In this subsection we show that it is necessary to work modulo elementary equivalence in Theorem~A, and that we cannot require the henselian local domains to be noetherian.

\begin{fact}\label{fact:sharp}
Suppose that one of the following holds.
\begin{enumerate}[leftmargin=*]
\item $K$ embeds into $\Rr$.
\item $K$ is an algebraic extension of a finite field.
\end{enumerate}
Then $K$ is not isomorphic to the fraction field of a proper henselian local domain \textup{(}i.e., a henselian local domain that is not a field\textup{)}.
\end{fact}

\begin{proof}
For (1), suppose towards a contradiction that $R \subseteq \Rr$ is a proper henselian local domain.
Let $\beta$ be a non-zero element of the maximal ideal of $R$.
After replacing $\beta$ with a suitable integer multiple (possibly negative) we have $\beta > 1/4$.
Then $x^2 - x + \beta$  does not have a root in $R$, which contradicts Hensel's lemma.

\medskip
For case (2), 
note that any subring $R \subseteq \Ff_p^{\alg}$ is a field, because for each $k < \infty$, the intersection $R \cap \Ff_{p^k}$ is a finite integral domain and hence a field.
\end{proof}

In particular, if $K$ is large and satisfies (1) or (2), then $K$ is elementarily equivalent to a fraction field of a proper henselian local domain, but not isomorphic to one.
For example, an infinite algebraic extension of a finite field is $\pac$ and hence large~\cite[Cor.~11.2.4]{field-arithmetic}.
Examples of large subfields of $\Rr$ include any real closed subfield (e.g., $\Rr$ itself) and the field of totally real algebraic numbers~\cite[Thm.~$\mathfrak{S}$]{pop-embedding}.

\medskip
We recall one more example.
Given $0 < r < 1$ let $\Cc_r[[t]]$ be the ring of holomorphic functions on the open disc of radius $r$ that extend to continuous functions on the closed disc of radius $r$.
Let $\Zz_r[[t]]$ be the subring of $\Cc_r[[t]]$ of functions whose Taylor series has coefficients in $\Zz$.
Then $\Frac \Zz_r[[t]]$ is large and is not the fraction field of a henselian local domain~\cite[Remark~3.3]{arithmetic-power-series}.

\medskip
Can we take the henselian local domain in Theorem~A to be noetherian?
The next lemma shows this fails for $\Rr$.
This is presumably well-known, but we include a proof for the sake of completeness.

\begin{fact}\label{fact:rcf-noe}
Suppose that $K$ is real or algebraically closed.
Then $K$ cannot be the fraction field of a non-field noetherian domain.
\end{fact}

\begin{proof}
Otherwise, $K$ admits a non-trivial discrete valuation $K^\times \to \Zz$ by \cite[Lemma~00PH]{stacks-project}.
However $K^\times$ has a divisible subgroup of finite index, so no such valuation can exist.
\end{proof}

In fact, this generalizes to PRC fields:

\begin{lemma}\label{lem:noetherian}
A $\prc$ field cannot be the fraction field of a non-field noetherian domain
\end{lemma}

\begin{proof}
Recall that $\pac$ fields are $\prc$ by definition.
Reduce to the $\pac$ case by recalling that $K[\sqrt{-1}]$ is $\pac$ when $K$ is $\prc$ and $R[\sqrt{-1}]$ is noetherian when $R$ is noetherian.
Let $R$ be a non-field noetherian domain.
As in the proof of Fact~\ref{fact:rcf-noe} $K = \Frac(R)$ is the fraction field of a DVR, so $K$ admits a non-trivial discrete valuation $v$.
By Fact~\ref{fact:old-EO}(\ref{old:dvr}) the $\cE_K$-topology refines the $v$-topology.
Now apply~Fact~\ref{old:pac} below.
\end{proof}

\begin{fact}\label{old:pac}
If $K$ is $\pac$ then the $\cE_K$-topology cannot refine a Hausdorff~field~topology~on~$K$.
\end{fact}

Fact~\ref{old:pac} is not explicitly stated in \cite[Prop.~7.1]{field-top-2}, but is immediate from its proof.

\subsection{A proper morphism over pseudofinite $K$}
The \'etale-open topology is the coarsest system of topologies such that \'etale morphisms induce open maps.
If $V \to W$ is smooth then $V(K) \to W(K)$ is also $\cE_K$-open~\cite[Prop.~3.2]{secondpaper}, so the \'etale-open topology is also the coarsest system of topologies such that smooth morphisms induce open maps.
It is natural to ask if something similar happens for the $\cF_K$-topology, with respect to \emph{proper} morphisms.

\medskip
A morphism $V \to \Aa^n$ is finite if and only if it is proper and $V$ is affine; see Fact~\ref{fact:basic-f}(\ref{afp}).
Hence the $K$-rational image of $V \to \Aa^n$ is trivially $\Sa F_K$-closed when $V$ is affine and $V \to \Aa^n$ is proper.
It is also natural to ask if this generalizes to general proper $V \to \Aa^n$.
This is true in some cases:

\begin{proposition}
Suppose that $K$ is perfect, $f \colon V \to W$ is proper, and one of the following holds:
\begin{enumerate}[leftmargin=*]
\item $K$ is algebraically closed.
\item $K$ admits a non-trivial henselian valuation.
\item $K$ is t-henselian.
\end{enumerate}
Then $f(V(K))$ is both $\cE_K$-closed and $\cF_K$-closed.
\end{proposition}

Recall that the $\cF_K$-topology agrees with the $\cE_K$-topology by Theorem~\ref{thm:new-t-hensel}, so it is enough to show that $f(V(K))$ is $\cE_K$-closed.

\begin{proof}
Suppose $K$ is algebraically closed.
Then the $\cE_K$-topology is the Zariski topology by Fact~\ref{fact:old-EO}(\ref{old:sep}).
The image variety $f(V)$ is Zariski closed as $f$ is proper and the $K$-rational image of $f$ agrees with the set of $K$-points of $f(V)$ as $K$ is algebraically closed.

\medskip
Suppose that $K$ is not algebraically closed. 
Suppose furthermore that $K$ admits a non-trivial henselian valuation $v$.
By Fact~\ref{fact:old-EO}(5) the $\cE_K$-topology agrees with the $v$-topology.
By a theorem of Moret-Bailly $f(V(K))$ is $v$-closed~\cite[Thm.~1.3]{Moret_Bailly}.

\medskip
Case (3) follows from Moret-Bailly's result by elementary transfer.
In more detail, suppose that $K$ is t-henselian and note that $K$ is elementarily equivalent to a field $L$ admitting a henselian valuation $v$ \cite[Theorem~7.2]{Prestel1978}.
By a routine argument we may suppose that $L$ is an elementary extension of $K$.
By the proof of \cite[Remark~7.11]{Prestel1978}, the formulas that define the (unique) t-henselian topology $\uptau$ on $K$ also define the (unique) t-henselian topology on $L$, i.e., the $v$-topology.
By case (2), the $L$-rational image $f_L(V(L))$ is $v$-closed. 
As $K$ is an elementary substructure, the $K$-rational image $f(V(K))$ is also closed.
\end{proof}

We give an example showing that this does not generalize to large perfect fields.

\begin{proposition}
\label{prop:counter example}
Suppose that one of the following holds.
\begin{enumerate}[leftmargin=*]
\item $K$ is pseudofinite of characteristic not two.
\item $K$ is $\pac$ of characteristic not two and some $\beta \in K$ is not a square.
\item $K$ is $\prc$ and admits at least two distinct field orders.
\end{enumerate}
Then there is a proper morphism $f \colon V \to \Aa^1$  such that $f(V(K)) = K^\times$.
\end{proposition}

We let $\Pp^n$ be $n$-dimensional projective space over $K$.

\begin{proof}
It is easy to see that (1) is a case of (2).
Suppose we are in case (2).
Let $t$ be a coordinate on $\Aa^1$ and $x,y,z,w$ be homogeneous coordinates for $\Pp^3$.
Let $V$ be the subvariety of $\Aa^1 \times \Pp^3$ given by
\begin{align*}
yw &= x^2 - \beta w^2\\
z^2 &= \beta y^2 + tw^2
\end{align*}
We let $f$ be the projection $V \to \Aa^1$.
Then $f$ is proper as $\Pp^3$ is proper.
Let $V_p$ be the fiber of $V$ above $p \in \Aa^1$.
We need to show that if $p \in K$ then $V_p(K)$ is empty if and only if $p = 0$.

\medskip
We show that $V_0(K)$ is empty.
Now $V_0$ is the closed subvariety of $\Pp^3$ given by $yw = x^2 - \beta w^2$ and $z^2 = \beta y^2$.
Suppose that $V_0(K)\ne 0$.
Then there are $a,b,c,d \in K$, not all zero, such that $bd = a^2 - \beta d^2$ and $c^2 = \beta b^2$.
As $\beta$ is not a square we have $b = c = 0$, hence $a^2 - \beta d^2 = 0$, hence $a^2 = \beta d^2$.
For the same reason $a = d = 0$, giving a contradiction.

\medskip
We now fix $\alpha \in K^\times$ and show that $V_\alpha(K)\ne\emptyset$.
Let $U$ be the open subvariety of $\Pp^3$ given by $w \ne 0$.
It is enough to show that $(V_\alpha \cap U)(K)\ne \emptyset$.
As $K$ is $\pac$ it is enough to show that $V_\alpha \cap U$ is geometrically integral.
We identify $U$ with $\Aa^3$ in the natural way and therefore identify $V_\alpha \cap U$ with the subvariety of $\Aa^3$ given by
\begin{align*}
y &= x^2 - \beta \\
z^2 &= \beta y^2 + \alpha
\end{align*}
Now $V_\alpha \cap U$ is geometrically integral if and only if $\kalg[x,y,z]/( y - x^2 + \beta , z^2 - \beta y^2 - \alpha )$ is a domain.
Working in $\kalg$, we fix a square root $\beta^{1/2}$ of $\beta$ and make the substitutions $s = x/\beta^{1/2}$, $t = y/\beta$, and $u = z/\beta^{3/2}$ to see that $W$ is isomorphic to the $\kalg$-variety given by $t = s^2 - 1, u^2 = t^2 + \alpha/\beta^3$.
Let $\gamma = \alpha/\beta^3$, so $\gamma \ne 0$.
Thus it suffices to show that $\kalg[s,t,u]/(t - s^2 + 1, u^2 - t^2 - \gamma)$ is a domain.
This ring is isomorphic to the extension of $\kalg[s]$ given by adjoining a square root of $(s^2 - 1)^2 + \gamma$.
It is therefore enough to show that $(s^2 - 1)^2 + \gamma$ is not a square in $\kalg[s]$.\footnote{Let $R$ be an integrally closed integral domain (such as $\kalg[s]$) and let $b \in R$ be a non-square.
Then $R[x]/(x^2-b)$ is an integral domain.
To see this, let $L = \Frac(R)$.
Use the fact that $R$ is integrally closed to see that $b$ is not a square in $L$.
Then the polynomial $x^2-b$ is irreducible over $L$, so $L[x]/(x^2-b)$ is a field.
The ring $R[x]/(x^2-b)$ embeds into $L[x]/(x^2-b)$.}
Let $h(s) = (s^2 - 1)^2 + \gamma$.
Suppose that $h$ is a square.
Then any root of $h$ is also a root of $h'$.
We have $h'(s) = 2(s^2 - 1)(2s) = 4s(s - 1)(s + 1)$, so $h'(s)$ has roots $0,1,-1$.
We have $h(1) = \gamma = h(-1)$, so $1,-1$ are not roots of $h$.
Hence the only root of $h$ in $\kalg$ is $0$, so $h(s)$ is of the form $\lambda s^4$ for $\lambda \in \kalg$, which is clearly false.  This completes the proof in case (2).

\medskip
Finally, suppose that $K$ satisfies (3).
Fix $\beta \in K$ which is positive with respect to one field order and negative with respect to another.
Then $\beta$ is not a square in $K$.
Let $L = K(\sqrt{-1})$.
Then $L$ is $\pac$ of characteristic zero and $\beta$ is not a square in $L$.
Let $\Aa^1_L$ be the affine line over $L$ and let $V \to \Aa^1_L$ be the proper $L$-variety morphism defined above with $L$-rational image $L^\times$.
Let $W \to \Aa^2$ be the Weil restriction of $V \to \Aa^1_L$.
Then $W \to \Aa^2$ is a morphism of $K$-varieties with $K$-rational image $K^2 \setminus \{(0,0)\}$.
By \cite[4.10.1]{weil_reminders} $W \to \Aa^2$ is proper.
Let $\Aa^1 \to \Aa^2$ be the inclusion along the $x$-axis; this is a closed immersion and hence proper.
The pullback of $W \to \Aa^2$ by $\Aa^1 \to \Aa^2$ is a proper morphism with $K$-rational image $K^\times$.
\end{proof}

Suppose that $K$ is pseudofinite of characteristic not two.
Then $K$ is perfect and bounded.
By Proposition~\ref{prop:counter example} there is a proper morphism $f \colon V \to \Aa^1$ with $f(V(K)) = K^\times$.
But $K^\times$ is not closed in the $\cE_K$- or $\cF_K$-topologies on $K$ (which agree), because these topologies are non-discrete.

\section{Open problems} \label{sec:open}
We end with a number of open problems.

\begin{question} \label{discrete}
Can the finite-closed topology on a perfect field be discrete?
\end{question}

Such a field is non-Hilbertian by Proposition~\ref{prop:hilbertian} and non-large by Theorem~C(1).
Thus a positive answer to Question~\ref{discrete} in characteristic zero gives a positive answer to Question~\ref{obnoxious}.
 
\begin{question} \label{obnoxious}
Is there a characteristic zero field which is neither large nor Hilbertian?
\end{question}

The answer should ``obviously'' be yes, and yet Question~\ref{obnoxious} seems out of reach.  Question~\ref{discrete} is also related to Podewski's conjecture:

\begin{conjecture}[Podewski]
Let $K$ be an infinite field which is model-theoretically minimal, i.e., any definable subset of $K$ is finite or cofinite. 
Then $K$ is algebraically closed.
\end{conjecture}

Podewski's conjecture was proven in positive characteristic by Wagner~\cite{wagner-minimal-fields}.
Suppose $K$ is a counterexample to Podewski's conjecture.
Then some non-constant polynomial map $f \colon K \to K$ is non-surjective.
The image of $f$ can't be finite, so it is cofinite by minimality.
Hence the $\cF_K$-topology is discrete by Remark~\ref{rem:lampe-discrete}.
Thus a negative answer to Question~\ref{discrete} in characteristic zero would prove Podewski's conjecture.

\begin{conjecture}[Koenigsmann~\cite{JK-slim}]\label{conj:K}
Any infinite bounded field is large.
\end{conjecture}

Theorem~C shows that if $K$ is perfect and bounded then $\cF_K$-topology is discrete if and only if $K$ is not large.
Thus the perfect case of Conjecture~\ref{conj:K} is true if and only if the bounded case of Question~\ref{discrete} has a negative answer.

\bibliographystyle{amsalpha}
\bibliography{ref}

\providecommand{\bysame}{\leavevmode\hbox to3em{\hrulefill}\thinspace}
\providecommand{\MR}{\relax\ifhmode\unskip\space\fi MR }
\providecommand{\MRhref}[2]{%
  \href{http://www.ams.org/mathscinet-getitem?mr=#1}{#2}
}
\providecommand{\href}[2]{#2}
\begin{thebibliography}{JTWY24}

\bibitem[BCR98]{real-algebraic-geometry}
J.~Bochnak, M.~Coste, and M-F. Roy, \emph{Real algebraic geometry}, Springer,
  1998.

\bibitem[BSF13]{open-problems-ample}
Lior Bary-Soroker and Arno Fehm, \emph{Open problems in the theory of ample
  fields}, Geometric and differential {G}alois theories, S\'{e}min. Congr.,
  vol.~27, Soc. Math. France, Paris, 2013, pp.~1--11. \MR{3203546}

\bibitem[BSGJ18]{bsgj}
Lior Bary-Soroker, Wulf-Dieter Geyer, and Moshe Jarden, \emph{Morphisms of
  varieties over ample fields}, Bull. Korean Math. Soc. \textbf{55} (2018),
  no.~4, 1023--1035. \MR{3845944}

\bibitem[Cha99]{Chatzidakis}
Zo\'{e} Chatzidakis, \emph{Simplicity and independence for pseudo-algebraically
  closed fields}, Models and computability ({L}eeds, 1997), London Math. Soc.
  Lecture Note Ser., vol. 259, Cambridge Univ. Press, Cambridge, 1999,
  pp.~41--61. \MR{1721163}

\bibitem[DWY25]{field-top-2}
Philip Dittmann, Erik Walsberg, and Jinhe Ye, \emph{When is the {\'e}tale open
  topology a field topology?}, Israel Journal of Mathematics \textbf{269}
  (2025), 67--102.

\bibitem[Efr72]{efroymson}
Gustave Efroymson, \emph{Henselian fields and solid $k$-varieties. ii},
  Proceedings of the American Mathematical Society \textbf{35} (1972), no.~2,
  362--366.

\bibitem[EP05]{EP-value}
Antonio~J. Engler and Alexander Prestel, \emph{Valued fields}, Springer
  Monographs in Mathematics, Springer-Verlag, Berlin, 2005. \MR{2183496}

\bibitem[FJ05]{field-arithmetic}
Michael~D. Fried and Moshe Jarden, \emph{Field arithmetic}, Springer Berlin
  Heidelberg, 2005.

\bibitem[FP11]{arithmetic-power-series}
Arno Fehm and Elad Paran, \emph{Galois theory over rings of arithmetic power
  series}, Advances in Mathematics \textbf{226} (2011), no.~5, 4183--4197.

\bibitem[Gro67]{EGA-IV-4}
A.~Grothendieck, \emph{\'{E}l\'{e}ments de g\'{e}om\'{e}trie alg\'{e}brique.
  {IV}. \'{E}tude locale des sch\'{e}mas et des morphismes de sch\'{e}mas
  {IV}}, Inst. Hautes \'{E}tudes Sci. Publ. Math. (1967), no.~32, 361.
  \MR{238860}

\bibitem[Hod93]{Hodges}
Wilfrid Hodges, \emph{Model theory}, Encyclopedia of mathematics and its
  applications, vol.~42, Cambridge University Press, 1993.

\bibitem[Jar03]{Jarden-pro-p}
Moshe Jarden, \emph{On ample fields}, Arch. Math. \textbf{80} (2003), no.~5,
  475--477.

\bibitem[JK10]{JK-slim}
Markus Junker and Jochen Koenigsmann, \emph{Schlanke {K}\"{o}rper (slim
  fields)}, J. Symbolic Logic \textbf{75} (2010), no.~2, 481--500. \MR{2648152}

\bibitem[JTWY24]{firstpaper}
Will Johnson, Chieu-Minh Tran, Erik Walsberg, and Jinhe Ye, \emph{The
  \'etale-open topology and the stable fields conjecture}, J. Eur. Math. Soc.
  (JEMS) \textbf{26} (2024), no.~10, 4033--4070. \MR{4768414}

\bibitem[JWY23]{field-top-1}
Will Johnson, Erik Walsberg, and Jinhe Ye, \emph{The \'{e}tale open topology
  over the fraction field of a {H}enselian local domain}, Math. Nachr.
  \textbf{296} (2023), no.~5, 1928--1937. \MR{4608837}

\bibitem[JY25]{cxf-paper}
Will Johnson and Jinhe Ye, \emph{Curve-excluding fields}, J. Eur. Math. Soc.
  (JEMS) (2025).

\bibitem[Kos16]{kosters}
Michiel Kosters, \emph{Images of polynomial maps on ample fields}, Funct.
  Approx. Comment. Math. \textbf{55} (2016), no.~1, 23--30. \MR{3549010}

\bibitem[Kuh04]{kuhlmann}
F.-V. Kuhlmann, \emph{On places of algebraic function fields in arbitrary
  characteristic}, Advances in Mathematics \textbf{188} (2004), 399--424.

\bibitem[Lam]{6820}
Philipp Lampe, \emph{Can a non-surjective polynomial map from an infinite field
  to itself miss only finitely many points?}, MathOverflow,
  URL:https://mathoverflow.net/q/6820 (version: 2013-11-06).

\bibitem[MB11]{Moret_Bailly}
Laurent Moret-Bailly, \emph{An extension of {G}reenberg’s theorem to general
  valuation rings}, Manuscripta Math. \textbf{139} (2011), no.~1–2,
  153–166.

\bibitem[Mon17]{Samaria-2017}
Samaria Montenegro, \emph{Pseudo real closed fields, pseudo {$p$}-adically
  closed fields and {${\rm NTP}_2$}}, Ann. Pure Appl. Logic \textbf{168}
  (2017), no.~1, 191--232. \MR{3564381}

\bibitem[Mum04]{red-book}
David Mumford, \emph{The red book of varieties and schemes}, Lecture Notes in
  Mathematics, vol. 1358, Springer Berlin / Heidelberg, 2004.

\bibitem[Mun00]{Munkres}
James Munkres, \emph{Topology}, second ed., Prentice Hall, 2000.

\bibitem[Nag75]{nagata-local}
Masayoshi Nagata, \emph{Local rings}, Robert E. Krieger Publishing Co.,
  Huntington, N.Y., 1975, Corrected reprint. \MR{0460307}

\bibitem[Poo17]{poonen-qpoints}
Bjorn Poonen, \emph{Rational points on varieties}, Graduate Studies in
  Mathematics, vol. 186, American Mathematical Society, Providence, RI, 2017.
  \MR{3729254}

\bibitem[Pop96]{pop-embedding}
Florian Pop, \emph{Embedding problems over large fields}, Ann. of Math. (2)
  \textbf{144} (1996), no.~1, 1--34. \MR{1405941}

\bibitem[Pop10]{Pophenselian}
\bysame, \emph{Henselian implies large}, Ann. of Math. (2) \textbf{172} (2010),
  no.~3, 2183--2195. \MR{2726108}

\bibitem[Pop14]{Pop-little}
\bysame, \emph{Little survey on large fields---old \& new}, Valuation theory in
  interaction, EMS Ser. Congr. Rep., Eur. Math. Soc., Z\"{u}rich, 2014,
  pp.~432--463. \MR{3329044}

\bibitem[Pre81]{prestel-prc}
Alexander Prestel, \emph{Pseudo real closed fields}, Set theory and model
  theory ({B}onn, 1979), Lecture Notes in Math., vol. 872, Springer, Berlin-New
  York, 1981, pp.~127--156. \MR{645909}

\bibitem[PW23]{with-anand}
Anand Pillay and Erik Walsberg, \emph{Galois groups of large simple fields},
  Model Theory \textbf{2} (2023), 357--380.

\bibitem[PZ78]{Prestel1978}
Alexander Prestel and Martin Ziegler, \emph{Model theoretic methods in the
  theory of topological fields.}, J. Reine Angew. Math. \textbf{0299\_0300}
  (1978), 318--341.

\bibitem[Qui62]{Quigley}
Frank Quigley, \emph{Maximal subfields of an algebraically closed field not
  containing a given element}, Proc. Amer. Math. Soc. \textbf{13} (1962),
  no.~4, 562.

\bibitem[Rob87]{robinson-edmund}
Edmund Robinson, \emph{The geometric theory of {$p$}-adic fields}, J. Algebra
  \textbf{110} (1987), no.~1, 158--172. \MR{904186}

\bibitem[Rob96]{a_course_in_the_theory_of_groups}
Derek J.~S. Robinson, \emph{A course in the theory of groups}, Springer New
  York, 1996.

\bibitem[San96]{Sander-thesis}
Tomas Sander, \emph{Effektive {A}lgebraische {G}eometrie \"uber nicht
  algebraisch abgeschlossenen {K}\"orpern}, 1996, PhD Thesis, Universit\"at
  Dortmund.

\bibitem[Sch94]{weil_reminders}
Claus Scheiderer, \emph{Some reminders on {W}eil restrictions}, Real and
  {\'E}tale Cohomology, Lecture Notes in Mathematics, vol. 1588, Springer
  Berlin Heidelberg, Berlin, Heidelberg, 1994, pp.~30--41.

\bibitem[Ser97]{Serre_gcoho}
Jean-Pierre Serre, \emph{Galois cohomology}, Springer Berlin Heidelberg, 1997.

\bibitem[Sie20]{Sierpinski}
Wacław Sierpiński, \emph{Sur une propriété topologique des ensembles
  dénombrables denses en soi}, Fund. Math. \textbf{1} (1920), no.~1, 11–16.

\bibitem[SS78]{counterexamples_in_topology}
Lynn~Arthur Steen and J.~Arthur Seebach, \emph{Counterexamples in topology},
  Springer New York, 1978.

\bibitem[{Sta}25]{stacks-project}
The {Stacks Project Authors}, \emph{\textit{Stacks Project}},
  \url{https://stacks.math.columbia.edu}, 2025.

\bibitem[SY25]{vxf-paper}
Micha\l{} Szachniewicz and Jinhe Ye, \emph{Hyperbolicity and model-complete
  fields}, Int. Math. Res. Not. IMRN (2025), no.~4, Paper No. rnaf019, 21.
  \MR{4862093}

\bibitem[Wag00]{wagner-minimal-fields}
Frank~O. Wagner, \emph{Minimal fields}, J. Symbolic Logic \textbf{65} (2000),
  no.~4, 1833--1835. \MR{1812183}

\bibitem[Wal22]{topological_proofs}
Erik Walsberg, \emph{Topological proofs of results on large fields}, C.R. Math.
  \textbf{360} (2022), no.~G11, 1187–1192.

\bibitem[Wei82]{Weissauer}
Rainer Weissauer, \emph{Der hilbertsche irreduzibilitätssatz.}, Journal für
  die reine und angewandte Mathematik \textbf{334} (1982), 203--220 (ger).

\bibitem[WY23]{secondpaper}
Erik Walsberg and Jinhe Ye, \emph{\'{E}z fields}, J. Algebra \textbf{614}
  (2023), 611--649. \MR{4499357}

\end{thebibliography}
\end{document}